\documentclass{article}

\usepackage[sort,compress]{cite}
\usepackage{graphicx} 
\usepackage{tikz}
\usetikzlibrary{arrows.meta}
\usepackage{pifont}
\usepackage{float}
\usepackage{titlesec}
\usepackage{url}
\usepackage{amsmath,amssymb,amsthm} 
\usepackage{enumitem}               
\usepackage{times}                  
\usepackage{hyperref}
\usepackage{cleveref}
\usepackage[marginal]{footmisc}

\hypersetup{colorlinks=true,
linkcolor=blue,
citecolor=blue,
pdfborder={0 0 0}
}

\newcommand{\catchline}[5]{}

\providecommand{\keywords}[1]{\paragraph*{Keywords:} #1}

\providecommand{\rhd}{\triangleright}
\title{Quandle coloring quivers of pretzel links\fontsize{10pt}{1}}
\author{Qinghui Meng, Ximin Liu and Boxin Zhou\fontsize{8pt}{1}} 
\newenvironment{romanlist}[1][]{\begin{enumerate}[label=(\roman{enumi})]}{\end{enumerate}}
\date{}
\newtheorem{theorem}{Theorem}
\newtheorem{lemma}[theorem]{Lemma}
\newtheorem{proposition}[theorem]{Proposition}
\newtheorem{corollary}[theorem]{Corollary}
\theoremstyle{definition}
\newtheorem{definition}[theorem]{Definition}
\newtheorem{example}[theorem]{Example}

\begin{document}
\maketitle 
\begin{abstract}
In this paper, we conduct a systematic study of quandle colorings and quandle coloring quivers for pretzel links using the dihedral quandle $\mathbb{Z}_{n}$. First, we systematically investigate all possible colorings of 3-pretzel links, determining the number of distinct colorings in each case as well as the structure of their quandle coloring quivers. In order to obtain more general conclusions, we impose restrictions on $n$ based on the properties of the coefficient matrix of the system of congruence equations. So we examine the number of quandle colorings and the quandle coloring quivers for 4-pretzel links in the case where $n$ is prime. Finally, Combining the results of 4-pretzel links we rigorously derive both the coloring numbers and quandle coloring quivers for general $m$-pretzel links in the case where $n$ is prime, with full proofs provided.
\end{abstract}
\section{General Appearance}	
	The study of knot theory has long been a central topic in topology, and quandles, as algebraic structures, provide a powerful tool for understanding its properties. The concept of quandle was introduced by Joyce \cite{Joyce1982} in 1982. By abstracting the Reidemeister moves from knot theory into algebraic axioms, he formulated the definition of quandle and applied it to the classification of knots. With the generalization of quandles and the introduction of homology theory, numerous knot invariants have been developed. In 2015, Elhamdadi and Nelson provided a comprehensive survey and exposition of these advancements in \cite{Mohamed2015}. In 1972, Gabriel introduced the concept of quivers and their representations in \cite{Gabriel1972}. In 2018, Cho and Nelson combined quiver structures with quandle colorings and proposed the notion of quandle coloring quivers in \cite{Cho2018}. In recent years, several scholars have conducted research on the quandle coloring quivers of torus links using dihedral quandles in \cite{Basi2021,Basi2021p2,Zhou2023,Zhou2024,Elhamdadi2025}.\par 
	As a result, the study of quandle coloring quivers of torus links has reached a significant level of maturity in recent literature. In 2003, Brownell investigated the number of fundamentally distinct $m$-colorings for $(p,q,r)$-pretzel knots in \cite{Brownell2003}. Subsequently, in 2013, Ostrander examined $p$-colorings of 3-pretzel knots and 4-pretzel knots in \cite{Ostrander}. However, research on colorings of pretzel links has thus far been limited to $p$-colorings.\par 
	Therefore, this paper will focus on quandle colorings and quandle coloring quivers of pretzel links using $dihedral\ quandle$ $\mathbb{Z}_{n}$. We first investigate the classical 3-pretzel links. To generalize our findings, we impose parameter restrictions and employ inductive reasoning to deduce the coloring properties of general pretzel links. We study the coloring situations of the 4-pretzel links when $n$ is a prime number, and then generalize the conclusion to the coloring situations of the general pretzel links and provide a proof. Finally, according to the coloring numbers of different types, the corresponding quiver structures are given.

	\section{Preliminary Knowledge}
	We first review some basic concepts of quandle and graph which in \cite{Mohamed2015,Basi2021}.
	\begin{definition}[Quandle]
		A quandle is a set $X$ with a binary operation $\rhd:X\times X\rightarrow X\ $satisfying:\\
		(1) For all $x\in X, x\rhd x=x.$\\
		(2) For all $y\in X$, the map $\beta_{y}:X\rightarrow X$ defined by $\beta_{y}(x)=x\rhd y\ $is invertible.\\
		(3) For all $x,y,z\in X,(x\rhd y)\rhd z=(x\rhd z)\rhd (y\rhd z)$.
	\end{definition}
	\begin{example}
		Given an oriented link $K$ and its link diagram $D$, we label each arc of the diagram with different symbols (such as $x, y, z$) and have the following relationships in Fig.\ref{t1} at the crossings. The fundamental quandle is the set of these symbols and their relationships, denoted as $\mathcal{Q}(K)$.
		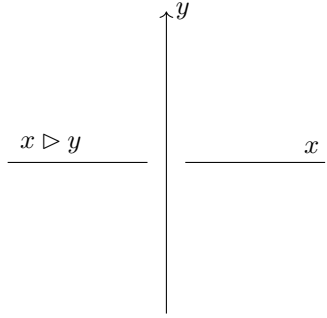
\begin{figure}[H]
			\centering
			\begin{tikzpicture}
				\draw[-] (-2.1,0) -- (-0.25,0) node[above] {$x\rhd y\quad \qquad\qquad\qquad\;$};
				\draw[-] (0.25,0) -- (2.1,0) node[above] {$x\quad$};
				\draw[->] (0,-2) -- (0,2) node[right] {$y$};
			\end{tikzpicture}
			\caption{Crossing relation}
			\label{t1}
		\end{figure}
	\end{example}
	\begin{example}
		Let $X=\mathbb{Z}_{n}$ and define $x \rhd y\equiv 2y - x\ (mod\ n)$ for any $x, y \in \mathbb{Z}_{n}$. Then we call $X$ is a $dihedral\ quandle$. Furthermore, the binary operation of the dihedral quandle is its own inverse.
	\end{example}
	\begin{definition}[$Quandle\ homomorphism$]
		Let $(X,\rhd _{1})$ and $(Y,\rhd _{2})$ be two quandles, a quandle homomorphism is a map $f:X\rightarrow Y$ satisfying $f(x\rhd_{1}y) =
		f(x)\rhd _{2}f(y)$, for all $x,y\in X$.
	\end{definition}
	\begin{definition}[Quandle coloring]
		Let $X$ be a finite quandle, $K$ an oriented link and $f:\mathcal{Q}(K)\rightarrow X $ be a quandle homomorphism, the coloring
		space $Hom(\mathcal{Q}(K), X)$ is the set of all quandle homomorphisms from $\mathcal{Q}(K)$ to $X$, and the cardinality of this set is denoted by $|Hom(\mathcal{Q}(K),X)|$, it is a quandle counting invariant, which we denote by $\Phi_{X}^{\mathbb{Z}}(K)$.
	\end{definition}
	\begin{definition}[$Directed\ multigraph$]
		A directed multigraph $\vec{G}$ is an ordered triple $(V,E,\Phi)$, and $\Phi:E\rightarrow \{(x,y):x,y\in V,x\neq y\}$ is an incidence function that maps edges of E to ordered pairs of vertices in $V$. In a concrete graphical representation, edges are typically marked with arrows to indicate their direction.
	\end{definition}
	\begin{definition}[$Directed\ pseudograph$]
		The directed pseudograph $\overleftrightarrow{D}$ is a generalization of an oriented graph, where loops and repeated edges are allowed. To capture multiple edges between a pair of vertices, we replace the incidence function $\Phi$ with a weight function $c:\{(x,y);x,y\in V\}\rightarrow \mathbb{N}\cup{0}$, which assigns a number to represent the total count of edges between any two vertices. We define this as $(\overleftrightarrow{D},c)$. When $\overleftrightarrow{D}$ has $n$ vertices and $c$ is a constant function, i.e., for all $x,y\in V$, there exists some $k\in \mathbb{N}$ such that $c(x,y)=k$, we define it as $(\overleftrightarrow{K_{n}},\hat{k})$ and refer to it as a complete graph.
	\end{definition}
	\begin{definition}
		Given two graphs $G_{1}$ and $G_{2}$ with disjoint vertex sets and edge sets, when there are $c$ edges directed from every vertex of $G_{2}$ to every vertex of $G_{1}$, we define this graph as $G_{1}\overset{\leftarrow}{\bigtriangledown}_{\hat{c}} G_{2}$.
	\end{definition}
	\begin{definition}[Quandle coloring quiver]
		Let $X$ be a finite quandle and $L$ an oriented link. For any set of quandle endomorphims $S\in Hom(X,X)$, the associated quandle coloring quiver, denoted $\mathcal{Q}^{S}_{X}(L)$, is the directed graph with a vertex for every element $f\in Hom(\mathcal{Q}(L),X)$ and an edge directed from $f$ to $g$ when $g= \phi\circ f$ for an element $\phi\in S$. When $S=Hom(X,X)$, which we call the $full\ quandle\ coloring\ quiver$ of $L$ with respect to $X$, denoted by $\mathcal{Q}_{X}(L)$.
	\end{definition}\par 
	From \cite{Ostrander,Kathryn,Kim2007} we can obtain these definitions and conclusions about pretzel links:
	\begin{definition}[Tangle]
		A tangle is a region of a knot diagram enclosed within a circle such that the knot passes through the circle at four points. Given two tangles $A$ and $B$, we denote their sum as $A+B$, formed by connecting the $NE$ endpoint of $A$ to the $NW$ endpoint of $B$, and the $SE$ endpoint of $A$ to the $SW$ endpoint of $B$, as illustrated in Fig.\ref{fig2.1}.
		\begin{figure}[H]
			\centering
			\includegraphics[width=0.52\linewidth]{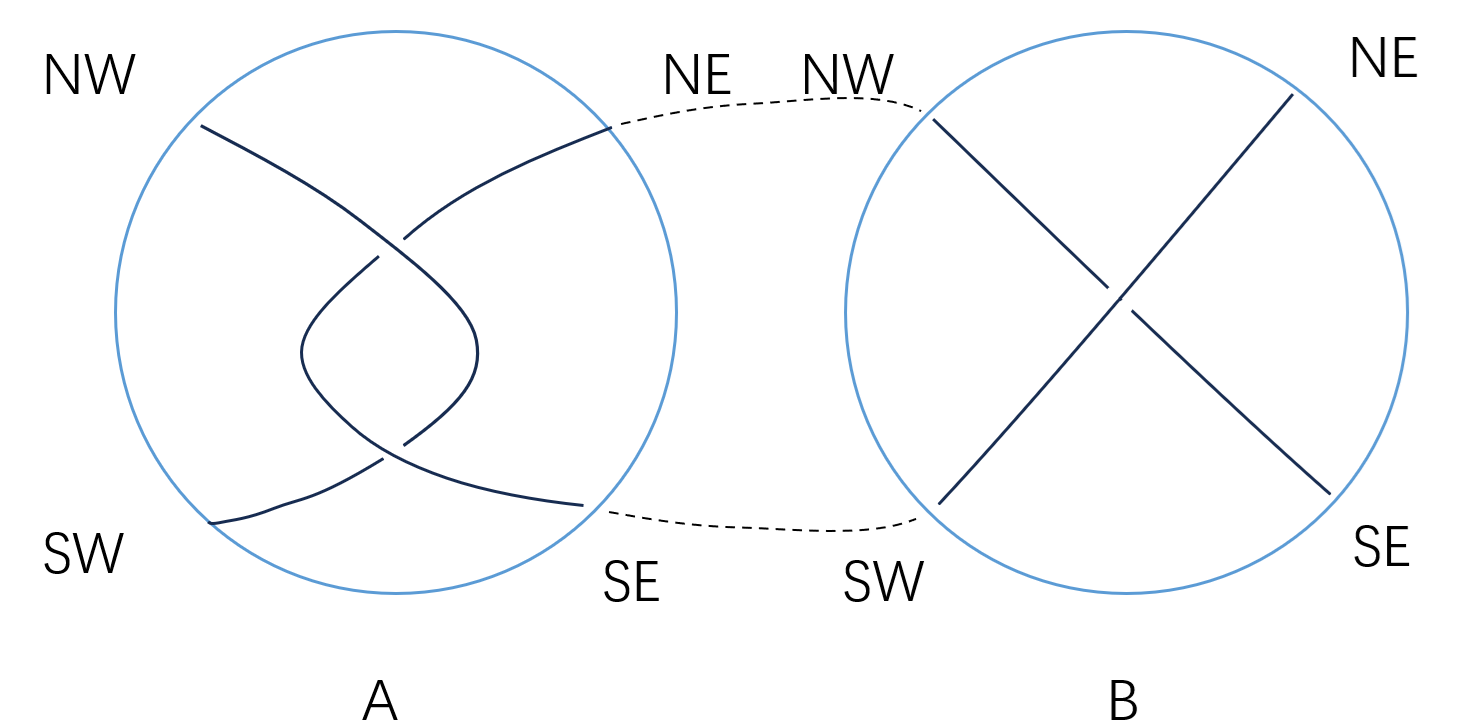}
			\caption{$tangle\ A+B$}
			\label{fig2.1}
		\end{figure}
	\end{definition}
	\begin{definition}[Pretzel knot or link]
		Given a positive integer $m$, nonzero integers $p_{1},p_{2},...,p_{m}$ and $tangles\ A_{i}$, where $A_{i}$ contains $|p_{i}|$ positive crossings when $p_{i}>0$ and $|p_{i}|$ negative crossings where $p_{i}<0$, the pretzel knot or link is defined as the sum $A_{1}+A_{2}+...+A_{m}$. For a general braided link as shown in Fig.\ref{t23}, we denote a pretzel link with $m$ strands as an $m$-pretzel link for brevity.
		\begin{figure}[H]
			\centering
			\includegraphics[width=0.5\linewidth]{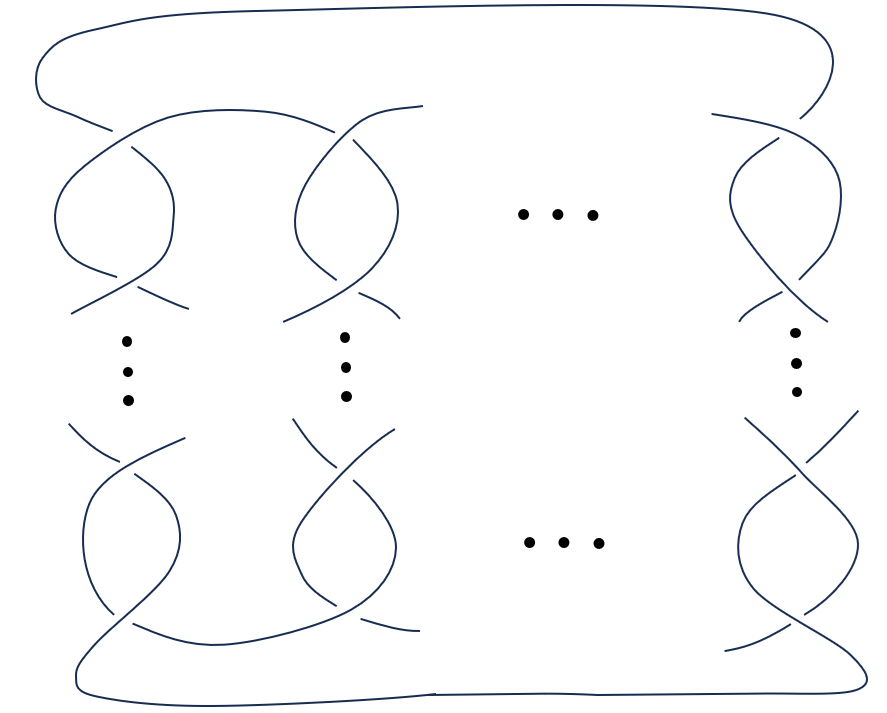}
			\caption{$m$-pretzel link}
			\label{t23}
		\end{figure}
	\end{definition}
	\begin{proposition}
		If $m$ is odd, then an n-pretzel link $(p_{1},p_{2},...,p_{m})$ is a knot if and only if none of two $p_{i}$’s are even. If $m$ is even, then $(p_{1},p_{2},...,p_{m})$ is a knot if and only if only one of the $p_{i}$’s is even. Generally the number of even $p_{i}$’s is the number of components, unless $p_{i}$’s are all odd.
	\end{proposition}
	\begin{theorem}\label{th2.13}
		For a given pretzel knot $(p, q, r)$, interchanging any two parameters among $p, q$, and $r$ results in an equivalent pretzel knot. This property extends to braid links as well, since the proof in \cite{Ostrander} is independent of the parity (evenness or oddness) of $p, q$, and $r$. Moreover, it is straightforward to verify that the conclusion holds for general pretzel links.
	\end{theorem}
	\begin{theorem}\label{th2.14}
		For a given pretzel link $(p_{1},p_{2},p_{3})$, its mirror image is $(-p_{1},-p_{2},-p_{3})$.
	\end{theorem}

	For a general system of linear congruence equations:
	\begin{equation}
		\begin{cases}
			a_{11}x_{1}+a_{12}x_{2}+...+a_{1n}x_{n}\equiv b_{1}\qquad(mod\ m_{1}),\\
			a_{21}x_{1}+a_{22}x_{2}+...+a_{2n}x_{n}\equiv b_{2}\qquad(mod\ m_{2}),\\
			\qquad \qquad......\\
			a_{s1}x_{1}+a_{s2}x_{2}+...+a_{sn}x_{n}\equiv b_{s}\qquad(mod\ m_{s}).
		\end{cases}
		\label{ty}
	\end{equation}
	Here $a_{ij},b_{i}$ and $m_{i}$ (where $1\leq i\leq s,1\leq j\leq n$) are integers, and $x_{1},x_{2},...,x_{n}$ are $n$ integer-valued variables. Typically, the $m_{i}(1\leq i\leq s)$ are distinct. Let $[a,b]$ denote the least common multiple (LCM) of $a$ and $b$, then let $[m_{1},m_{2},...m_{s}]=m_{1}^{\prime}m_{1}=m_{2}^{\prime}m_{2}=...=m_{s}^{\prime}m_{s}$ represent the LCM of $m_{1},m_{2},...,m_{s}$. Based on the properties of linear congruence equations and \cite{mathews1892}, we can conclude that the system of congruence equations (\ref{ty}) is equivalent to the following system (\ref{ty1}):
	\begin{equation}
		\begin{cases}
			m_{1}^{\prime}a_{11}x_{1}+m_{1}^{\prime}a_{12}x_{2}+...+m_{1}^{\prime}a_{1n}x_{n}\equiv m_{1}^{\prime}b_{1}\\
			m_{2}^{\prime}a_{21}x_{1}+m_{2}^{\prime}a_{22}x_{2}+...+m_{2}^{\prime}a_{2n}x_{n}\equiv m_{2}^{\prime}b_{1}\\
			\qquad \qquad......\\
			m_{s}^{\prime}a_{s1}x_{1}+m_{s}^{\prime}a_{s2}x_{2}+...+m_{s}^{\prime}a_{sn}x_{n}\equiv m_{s}^{\prime}b_{s}
		\end{cases}(mod\ [m_{1},m_{2},...,m_{s}]).
		\label{ty1}
	\end{equation}
	For a system of linear congruence equations where each equation is modulo $m$, we refer to it as a $modulo\ m\ system\ of\ linear\ congruence\ equations$. \par
	\begin{proposition}\textsuperscript{\cite{Wang}}\label{prop}
		Performing the following elementary row operations on a linear congruence system modulo $m$ results in an equivalent system with identical solutions.\par 
		(1)location operation: Interchanges any two rows of the matrix;\par
		(2)elimination operation: Adding a multiple of each elements of one row to the 
		corresponding elements of another row;\par 
		(3)multiple operation: multiplying an integer $k$ to each elements of one row, where $k$
		and the modular $m$ are coprime;\par 
		(4)modulo operation: applying the modulo operation to any entries of the modular matrix, i.e., adding a multiple of $m$ to an element or subtracting a multiple of $m$ from any element of the modular matrix, denoted this operation as $(mod\ m)$.
	\end{proposition}
	\begin{lemma}\label{lem2.16}
		For matrices over $\mathbb{Z}_{n}$, when $n$ is a prime number, it is always possible to reduce them to reduced row echelon form (RREF) through row operations. However, when $n$ is a composite number, they may not necessarily be reducible to RREF via row transformations.
	\end{lemma}
	\begin{proof}
		When $n$ is a prime number, for any integer $a$, either $\gcd(a,n)=1$ or $a$ is a multiple of $n$. Then for general matrices over $\mathbb{Z}_{n}$:
		\begin{equation}
			A=\begin{bmatrix}
				a_{11}&a_{12}&\cdots &a_{1n}\\
				a_{21}&a_{22}&\cdots &a_{2n}\\
				\vdots&\vdots&\ddots &\vdots\\
				a_{s1}&a_{s2}&\cdots &a_{sn}
			\end{bmatrix}.
			\nonumber
		\end{equation} 
		If $A=0$, it is trivial. If $A\neq 0$, it must contain a element $a_kl\neq 0$, which is also coprime with $n$. \par
		Find the first nonzero element in the first column. If $a_{11}\neq 0$, then $a_{11}$ is the first pivot. If $a_{11}=0$ and there exists some $a_{k1}\neq 0,k=2,...,s$, then swap the first row with the $k$-th row, making $a_{k1}$ the pivot. If no such nonzero element exists in the first column, we take no action and proceed to the next column for further operations.Without loss of generality, assume $a_{11}\neq 0$, then $a_{11}$ is a unit (invertible element) in $\mathbb{Z}_{n}$, i.e., there exists $a_{11}^{-1}\in \mathbb{Z}_{n}$ such that $a_{11}a_{11}^{-1}=1$. Clearly, $a_{11}^{-1}$ is also a unit (invertible element) in $\mathbb{Z}_{n}$. By multiplying every element of the first row by $a_{11}^{-1}$, the pivot can be normalized. If $a_{21}\neq 0$, then subtract $a_{21}$ times each element of the first row from the corresponding element of this row, and proceed analogously to eliminate all elements below the pivot in the first column.\par 
		For the second column: If there exists a nonzero element below $a_{12}$, proceed as follows (assuming without loss of generality $a_{22}\neq 0$). Multiply every element of the second row by $a_{22}^{-1}$ to normalize its pivot entry to 1. Then, using analogous row operations, eliminate all other entries in the second column.If no such nonzero element exists in the second column, we take no action and proceed to the next column for further operations.\par 
		This process is repeated iteratively to transform the entire matrix into RREF via elementary row operations.\par 
		However, when $n$ is composite: Pivot elements may not be invertible (since nonzero elements in $\mathbb{Z}_{n}$ need not be units). Elements in the same column may not be exact multiples of the pivot (due to zero divisors).
		Thus, the matrix cannot always be reduced to RREF.
	\end{proof}
	Let the coefficient matrix of the congruence equation system be $A=\begin{bmatrix}
		a_{11}&a_{12}&\cdots &a_{1n}\\
		a_{21}&a_{22}&\cdots &a_{2n}\\
		\vdots &\vdots &\ddots &\vdots\\
		a_{s1}&a_{s2} &\cdots &a_{sn}
	\end{bmatrix}$. Let $b=\begin{bmatrix}
		b_{1}\\
		b_{2}\\
		\vdots\\
		b_{s}
	\end{bmatrix}$ and $x=\begin{bmatrix}
		x_{1}\\
		x_{2}\\
		\vdots\\
		x_{n}
	\end{bmatrix}$, then the system of linear congruences modulo $m$ can then be expressed as:
	\begin{equation}
		Ax\equiv b\ (mod\ m).
		\label{2.1}
	\end{equation}
	\begin{theorem}\label{th2.17}\textsuperscript{\cite{Wang}}
		Let $A$ be an $s\times n$ matrix modulo $m$, and let $P$ and $Q$ be invertible matrices of sizes $s\times s$ and $n\times n$, respectively. Then, the congruence system (\ref{2.1}) is equivalent to the congruence system (\ref{2.2}):
		\begin{equation}
			PAQy\equiv Pb\ (mod\ m),
			\label{2.2}
		\end{equation}
		where $y=\begin{bmatrix}
			y_{1}\\
			y_{2}\\
			\vdots\\
			y_{n}
		\end{bmatrix}$. When the congruence system has solutions, the relationship between the solutions of the two systems is given by:
		\begin{equation}
			x=Qy\ (mod\ m).
			\nonumber
		\end{equation}
	\end{theorem}
	
	\begin{theorem}\textsuperscript{\cite{Florentin}}
		The linear congruence equation $a_{1}x_{1}+...+a_{n}x_{n}\equiv b\ (mod\ m)$ has a solution if and only if $(a_{1},...,a_{n},m)|b$.
	\end{theorem}
	\begin{theorem}\textsuperscript{\cite{Florentin}}\label{th2.19}
		The linear congruence equation $a_{1}x_{1}+...+a_{n}x_{n}\equiv b\ (mod\ m)$ has $d|m|^{n-1}$ distinct solutions modulo $m$ when $m\neq 0,\ \gcd(a_{1},...,a_{n},m)=d|b$.
	\end{theorem}
	
	
	\section{Quandle coloring spaces of pretzel links}
	To simplify the expression, the following linear congruent equations and systems of linear congruent equations are all modulo $n$ unless otherwise specified.

	\subsection{Quandle coloring spaces of 3-pretzel links}
	By Theorem \ref{th2.13} and Theorem \ref{th2.14}, there are only two forms of the classic pretzel links, i.e., $P_1^1=(p_ {1}, p_{2},p_{3})$ and $P_2^3=(p_{1}, p_{2},-p_{3})$, where $p_{1},p_{2},p_{3}\in \mathbb{Z}^{+}$. Let's start by study the numbers of coloring of the first form of the three-pronged pretzel links.\par 
	From Lemma \ref{lem2.16}, it can be known that for the coefficient matrix of a general linear congruence system of equations, it is very likely that it cannot be reduced to a RREF, and thus the explicit solution of the linear congruence system of equations cannot be obtained. For a linear congruent system of equations composed of two ternary equations, at least one of them must be a free element. Therefore, we can perform variable substitution, that is, subtract the free element from the remaining unknowns and study it as a new variable.
	\begin{theorem}\label{th3}
		Given an oriented pretzel link $P_1^3=(p_{1},p_{2},p_{3})$ and a dihedral quandle $\mathbb{Z}_{n}$, where $p_{1},p_{2},p_{3},n\in \mathbb{Z}^{+}$. Let $P=p_1p_2+p_1p_3+p_2p_3,Q=\gcd(p_{1},p_{2},p_3),D_{1}=\gcd(Q,n),D_2=\gcd(\frac{|P|}{Q},n)$. Then the coloring number of $P_1^3$ is
		\begin{equation}
			|Hom(P_1^3,\mathbb{Z}_n)|=
			\begin{cases}
				n\qquad \ \qquad \gcd(P,n)=1;\\
				D_1D_2n\qquad others.
			\end{cases}
			\nonumber
		\end{equation}
	\end{theorem}
	\begin{proof}
		Given an oriented pretzel link $P_1^3=(p_{1},p_{2},p_{3})$, as shown in below.
		\begin{figure}[H]
			\centering
			\includegraphics[width=0.4\textwidth,height=0.39\textwidth]{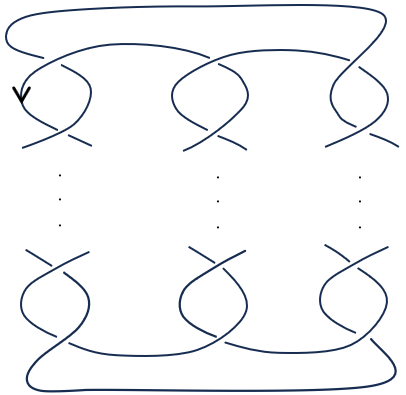}
			\caption{$P_1^3$}
			\label{t123}
		\end{figure}\par
		In Fig.\ref{t123}, there are many arcs without marked directions. Although different directions of those unmarked arcs will result in different fundamental quandles, since the binary operation of the dihedral quandle takes itself as the inverse, the direction of the arc will not change the dihedral quandle of $P_1^3$.\par 
		Take any orientation of $P_1^3$, and let its fundamental quandle be $\mathcal{Q}(P_1^3)$, for any $f\in Hom(\mathcal{Q}(P_1^3),\mathbb{Z}_{n})$ satisfying that for any $x,y\in  \mathcal{Q}(P_1^3),f(x\triangleright y)=f(x)\triangleright f(y)=2f(y)-f(x)$. Next, let's examine the image of $f$. First, we set the coloring condition of the arc in the upper left corner as $a$, and the coloring condition of the arc below it as $b$. Then, according to the operation rule of the dihedral quandle, the coloring condition of the arc below $a$ and $b$ is $a\triangleright b=2b-a$. Consequently, the arcs in the first column can be successively colored as $a,b,2b-a,3b-2a,...,p_{1}b-(p_{1}-1)a,(p_{1}+1)b-p_{1}a$. Suppose the arc linking columns two and three below is colored by $c$. This fixes the colors of all arcs in the third column:$p_{1}b-(p_{1}-1)a,c,2c-p_{1}b+(p_{1}-1)a,...,p_{3}c-(p_{3}-1)p_{1}b+(p_{3}-1)(p_{1}-1)a,(p_{3}+1)c-p_{3}p_{1}b+p_{3}(p_{1}-1)a$. Since each arc can be assigned only one color, the coloring condition of the topmost arc implies:
		\begin{equation}
			p_{3}c-(p_{3}-1)p_{1}b+(p_{3}-1)(p_{1}-1)a\equiv a.
			\label{el1}
		\end{equation}
		The coloring condition of the arcs connecting the second and third columns above is as follows:
		\begin{equation}
			(p_{3}+1)c-p_{3}p_{1}b+p_{3}(p_{1}-1)a\equiv c+p_{1}(a-b).
			\nonumber
		\end{equation}
		Perform coloring on the intermediate arcs, similarly, we have:
		\begin{equation}
			\begin{cases}
				p_{2}c+p_{2}p_{1}(a-b)-(p_{2}-1)b\equiv (p_{1}+1)b-p_{1}a, \\
				(p_{2}+1)c+(p_{2}+1)p_{1}(a-b)-p_{2}b\equiv c.
			\end{cases}
			\label{3.2}
		\end{equation}
		The coloring of the entire $P_1^3$ using the dihedral quandle $\mathbb{Z}_n$ is illustrated as follows:
		\begin{figure}[H]
			\centering
			\includegraphics[width=1\textwidth,height=0.43\textwidth]{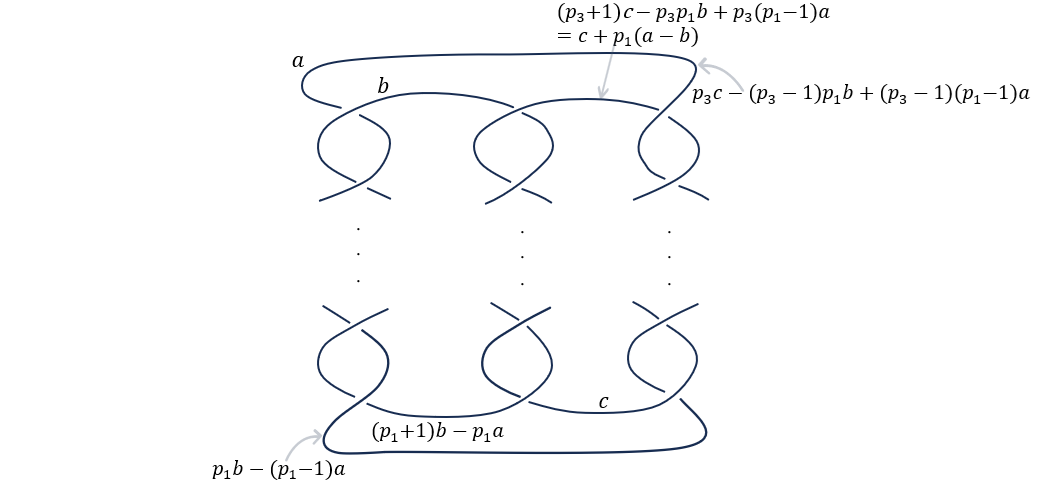}
			\caption{Coloring diagram of $P_1^3$ using the dihedral quandle $\mathbb{Z}_{n}$}
		\end{figure}\par  
		Combining these results, we conclude that the image of $f$ consists of linear combinations of the three elements $a,b$, and $c$. Let $f(x_{1})=a,f(x_{2})=b,f(x_{3})=c,x_{1},x_{2},x_{3}\in \mathcal{Q}(P_1^3)$, and $a,b,c\in \mathbb{Z}_{n}$. By simplifying equations (\ref{el1}) and (\ref{3.2}), we obtain a system of linear congruence equations describing the coloring conditions of $P_1^3$:
		\begin{equation}
			\begin{cases}
				p_{1}(p_{3}-1)(a-b)+p_{3}(c-a)\equiv 0, \qquad (1a)\\
				p_{1}(p_{2}+1)(a-b)+p_{2}(c-b)\equiv 0. \qquad \, (1b)
			\end{cases} 
			\label{el}
		\end{equation}
		We call this the $P_1^3$ coloring system, and now proceed to find its solutions.\par 
		
		We set
			\begin{equation}
				\begin{cases}
					a-b=x,\\
					c-b=y,\\
					c=z.
				\end{cases} 
				\nonumber 
			\end{equation}
			Expressed in matrix form:
			\begin{equation}
				\left[
				\begin{array}{c}
					x \\
					y \\
					z
				\end{array}
				\right]
				=
				\left[
				\begin{array}{ccc}
					1 & -1 &0 \\
					0 &-1 &1 \\
					0& 0 &1
				\end{array}
				\right]
				\left[
				\begin{array}{c}
					a \\ 
					b \\
					c
				\end{array}
				\right]
				=
				C
				\left[
				\begin{array}{c}
					a \\ 
					b \\
					c
				\end{array}
				\right].
				\nonumber 
			\end{equation}
			Since $det(C)=-1$, it follows from \cite{Stein2017} that $\gcd(-1,n)=\gcd(1,n)=1$. Therefore, the matrix $C$ is invertible, and the system of equations (\ref{el}) is equivalent to the system (\ref{el2}):
			\begin{equation}
				\begin{cases}
					(p_{1}p_{3}-p_1-p_3)x+p_{3}y\equiv 0, \quad  \\
					p_{1}(p_{2}+1)x+p_{2}y\equiv 0. \quad 
				\end{cases} 
				\label{el2}
			\end{equation}
			Expressed in matrix form:
			\begin{equation}
				\left[
				\begin{array}{cc}
					p_1p_3-p_1-p_3& p_3 \\
					p_1p_2+p_1& p_2
				\end{array}
				\right]
				\left[
				\begin{array}{c}
					x \\ 
					y
				\end{array}
				\right]
				\equiv
				A\mathbf{x}
				\equiv
				\mathbf{0}\pmod{n}.
				\label{eq5}
			\end{equation}\par 
			If the coefficient matrix $A$ is invertible, i.e., $\gcd(p_1p_2+p_1p_3+p_2p_3,n)=1$, then the homogeneous system has only the trivial solution, which implies $a\equiv b\equiv c$. In this case, only the trivial coloring exists, resulting in $|Hom(P_1^3,\mathbb{Z}_{n})|=n$.\par 
			If $\gcd(p_1p_2+p_1p_3+p_2p_3,n)=d\neq 1$, based on the prime factorization of the integer, there exist prime numbers $q_1, q_2,...q_s$ and integer $r_1,r_1,...,r_s$ such that $d = q_1^{r_1} q_2^{r_2}...q_s^{r_s}.(s\in \mathbb{Z}\ and\ s<n)$.\par
			We consider the system of equations (\ref{el2}) as a system of linear equations over the integral domain $\mathbb{Z}$. Over $\mathbb{Z}$, we factor the coefficient matrix $A$ as $A=UDV$, where: $U$ and $V$ are invertible integer square matrices (i.e., with determinant $\pm 1$), and D is the diagonal matrix $(D = \operatorname{diag}(d_1, d_2))$, which is the Smith normal form of matrix $A$, satisfying $d_1 \mid d_2$.\par
			The original equation $A \mathbf{x} \equiv \mathbf{0} \pmod{n}$ is equivalent to:  
			
			\begin{equation}
				UDV \mathbf{x} \equiv \mathbf{0} \pmod{n},
				\nonumber
			\end{equation}
			
			Since $U$ and $V$ are invertible over $\mathbb{Z}_n$, i.e., $|U|=|V|=±1$, taking determinants on both sides of $A=UDV$ gives:
			\begin{equation}
				(p_1p_2+p_1p_3+p_2p_3)=\pm d_1d_2.
			\end{equation}
			Then the problem reduces to how to solve for $d_1$.\par
			On the one hand, if we consider the prime factorization of D, then we have:
			\begin{itemize}
				\item[(1)] If $p_1p_2+p_1p_3+p_2p_3=q_1q_2$, and $q_1=1$ or $q_1$ is a prime number, $q_2$ is a prime number, then (\ref{el2}) is equivalent to
				\begin{equation}
					\begin{cases}
						x\equiv 0,\\
						d_1y\equiv 0.	
					\end{cases}
				\end{equation}
				It is easy to know that at this point we have $|Hom(P_1^3,\mathbb{Z}_n)|=dn$.
				\item[(2)] If $p_1p_2+p_1p_3+p_2p_3=q_1^2q_2...q_s$, where $q_1,q_2$ are prime numbers, then it is easy to know $d_1=q_1,\ d_2=q_1q_2...q_s$. Let $D_1=\gcd(d_1,n),\ D_2=\gcd(d_2,n)$, then (\ref{el2}) is equivalent to\par
				\begin{equation}
					\begin{cases}
						D_1x\equiv 0,\\
						D_2y\equiv 0.	
					\end{cases}
				\end{equation}
				It is easy to know that at this point we have $|Hom(P_1^3,\mathbb{Z}_n)|=D_1D_2n$.
			\end{itemize}
			On the other hand, according to the determinant divisor theory of the Smith normal form, we have the following two core relations:  \par
			\begin{itemize}
				\item[(1)] First-order determinant divisor $\Delta_1$ is the greatest common divisor of all $1\times1$ minors (i.e., all entries) of matrix $A$, and $d_1=\Delta_1$.  \par
				\item[(2)] Second-order determinant divisor $\Delta_2$ is the greatest common divisor of all $2\times2$ minors (i.e., determinants) of matrix $A$, and $\Delta_2=det(A)$.
			\end{itemize}
			It is easy to know $d_1=\gcd(p_1p_3-p_1-p_3,p_3,p_1(p_2+1),p_2)=\gcd(p_1,p_2,p_3)$, then we can get $$d_2=\frac{|p_1p_2+p_1p_3+p_2p_3|}{\gcd(p_1,p_2,p_3)}.$$
			Let $D_1=\gcd(d_1,n),D_2=\gcd(d_2,n)$, then we have $|Hom(P_1^3,\mathbb{Z}_n)|=D_1D_2n$.
			
	\end{proof}\par

	Compared with the large matrix computations in \cite{Ostrander}, it can be observed that quandle colorings are computationally simpler to work with. Moreover, quandle colorings constitutes a stronger invariant than $p$-coloring.
	\begin{example}\label{ex1}
		Given an oriented pretzel link $P_{246}=(2,4,6)$ and a dihedral quandle $\mathbb{Z}_{n}$. We can get $d_1=\gcd(p_1,p_2,p_3)=2,d_2=22$. When $n=4$, the quandle coloring number of $P_{246}$ is $2\times 2\times n=16$. When $n=8$, the quandle coloring number of $P_{246}$ is $2\times 6\times n=96$.
		\begin{proof}
			When $n=4$, substituting $p_{1}=2,p_{2}=4,p_{3}=6$ into the congruence system (\ref{el}) and simplifying yields:
			\begin{equation}
				\begin{cases}
					2(c-b)\equiv 0,\\
					2(a-c)\equiv 0.
				\end{cases}
				\nonumber
			\end{equation}
			
			It is straightforward to verify that in this case, the coloring number of $P_{246}$, $|Hom(P_{246},\mathbb{Z}_{4})|=2\times2\times n=16$.\par 
			When $n=8$, substituting $p_{1}=2,p_{2}=4,p_{3}=6$ into the congruence system (\ref{el}) and simplifying yields:
			\begin{equation}
				\begin{cases}
					2(b-c)\equiv 0,\\
					2(a-c)\equiv 0.
				\end{cases}
				\nonumber
			\end{equation}\par
				Clearly, in this case, the coloring number of $P_{246}$ is $|Hom(P_{246},\mathbb{Z}_{8})|=2\times 2\times n=32$.\par
		\end{proof}
	\end{example}
	\begin{corollary}
		Given an oriented pretzel link $P_2^3=(p_{1},p_{2},-p_{3})$ and a dihedral quandle $\mathbb{Z}_{n}$, where $p_{1},p_{2},p_{3},n\in \mathbb{Z}^{+}$. Let $P=p_1p_2-p_1p_3-p_2p_3,Q=\gcd(p_{1},p_{2},p_3),D_{1}=\gcd(Q,n),D_2=\gcd(\frac{|P|}{Q},n)$. Then the coloring number of $P_2^3$ is
		\begin{equation}
			|Hom(P_2^3,\mathbb{Z}_n)|=
			\begin{cases}
				n\qquad \ \qquad \gcd(P,n)=1;\\
				D_1D_2n\qquad others.
			\end{cases}
			\nonumber
		\end{equation}
	\end{corollary}
	\begin{proof}
		For any given oriented tangle link $P_2^3$, let $\mathcal{Q}(P_2^3)$ denote its fundamental quandle. Then for every 
		$f\in Hom(\mathcal{Q}(P_2^3),\mathbb{Z}_{n})$ following the same reasoning as in Theorem \ref{th3}, we can derive the coloring for each arc of $P_2^3$ using dihedral quandle as shown in the following diagram:
		\begin{figure}[H]
			\centering
			\includegraphics[width=0.85\textwidth,height=0.45\textwidth]{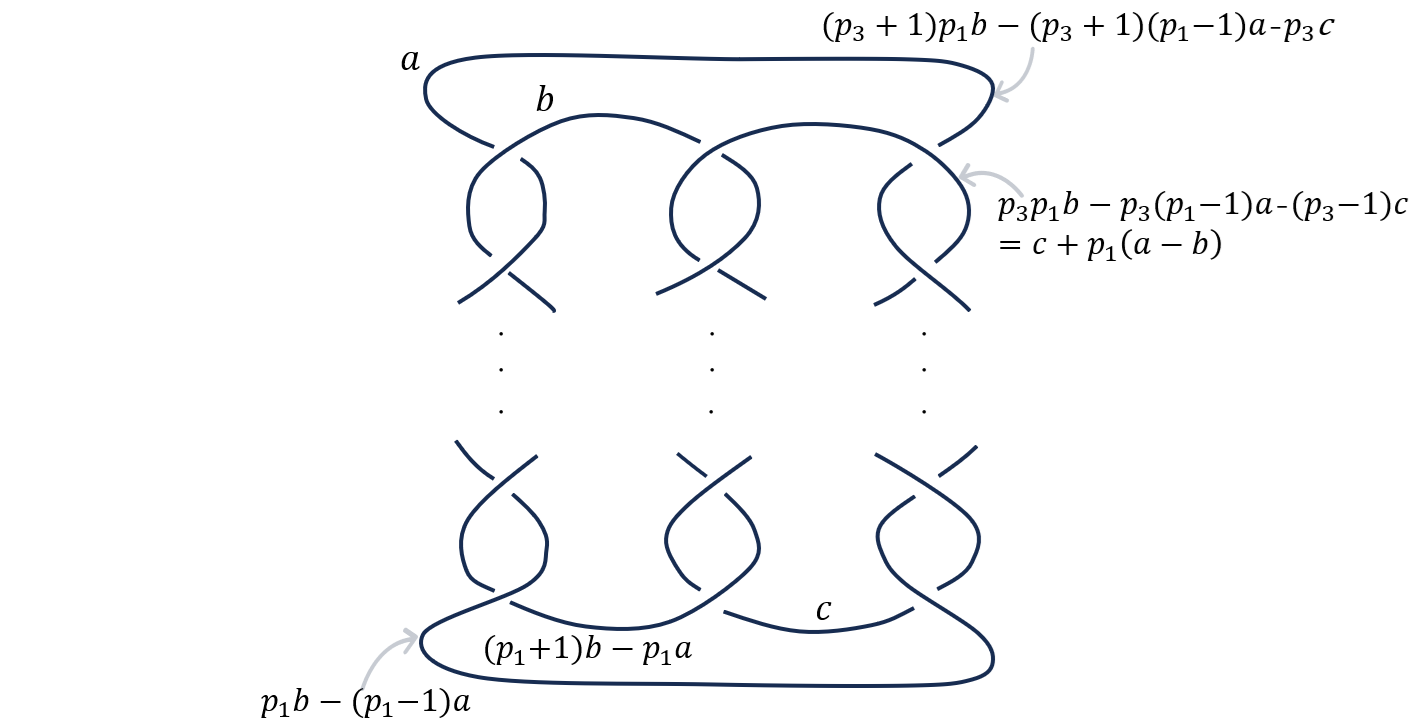}
			\caption{Coloring diagram of $P_2^3$ using the dihedral quandle $\mathbb{Z}_{n}$}
		\end{figure}\par
		And obtain the coloring equation system for $P_2^3$:
		\begin{equation}
			\begin{cases}
				p_{1}(p_{3}+1)(a-b)+p_{3}(c-a)\equiv 0,\\
				p_{1}(p_{2}+1)(a-b)+p_{2}(c-b)\equiv 0.
			\end{cases}
		\end{equation}
		Similarly to Theorem \ref{th3}, solving this system of equations yields the coloring number of $P_2^3$,
		\begin{equation}
			|Hom(P_2^3,\mathbb{Z}_n)|=
			\begin{cases}
				n\qquad \ \qquad \gcd(P,n)=1;\\
				D_1D_2n\qquad others.
			\end{cases}
			\nonumber
		\end{equation}
	\end{proof}

	\subsection{Quandle coloring spaces of 4-pretzel links} \label{sec3.2}
	According to \cite{Ostrander}, there are only four distinct types of 4-pretzel links, namely: $P_1^4=(p_{1},p_{2},p_{3},p_{4}),P_2^4=(p_{1},p_{2},p_{3},-p_{4}),P_3^4=(p_{1},p_{2},-p_{3},-p_{4})$ and $P_4^4=(p_{1},-p_{2},p_{3},-p_{4})$, and where $p_{1},p_{2},p_{3},p_{4}\in \mathbb{Z}^{+}$. We begin by examining the coloring space of the first type of 4-pretzel links.\par 
	For matrices over $\mathbb{Z}_{n}$, Lemma \ref{lem2.16} indicates that when $n$ is not a prime number, it may not be possible to transform them into the reduced row echelon form through row operations, thus preventing the derivation of explicit solutions. Consequently, the solution method for coloring equations of 3-pretzel links cannot be directly generalized to 4-pretzel links or more complex cases. When $n$ is prime, Lemma \ref{lem2.16} ensures the reducibility to reduced row echelon form over $\mathbb{Z}_{n}$. We therefore restrict our analysis of 4-pretzel links colorings to prime moduli.
	
	\begin{theorem}\label{th4p}
		Given an oriented 4-pretzel link $P_1^4=(p_{1},p_{2},p_{3},p_{4})$ and a dihedral quandle $\mathbb{Z}_{n}$, where n is a prime number, let $h=p_{1}p_{2}p_{3}+p_{1}p_{2}p_{4}+p_{1}p_{3}p_{4}+p_{2}p_{3}p_{4}$. Then the number of colorings of $P_1^4$ is 
		\begin{equation}
			|Hom(P_1^4,\mathbb{Z}_n)|=
			\begin{cases}
				n\qquad \ h\not\equiv 0;\\
				n^{2}\qquad h\equiv 0,\ and\ at\ least\ one\ p_{i}p_{j}\not \equiv 0\ hold\ for\ i\neq j,
				\\ \qquad \quad i,\ j =1,2,3,4;\\
				n^{3}\qquad  only\ one\ of\ p_{i},i=1,2,3,4\ is\ coprime\ with\ n;\\
				n^{4}\qquad p_{1}\equiv p_{2}\equiv p_{3}\equiv p_4\equiv 0.
			\end{cases}
			\nonumber
		\end{equation}
	\end{theorem}
	\begin{proof}
		For an arbitrary 4-pretzel link $P_1^4=(p_{1},p_{2},p_{3},p_{4})$ with all positive crossings, we denote its fundamental quandle as $\mathcal{Q}(P_1^4)$. For any $f\in Hom(\mathcal{Q}(P_1^4),\mathbb{Z}_{n})$, analogous to Theorem \ref{th3}, we first examine the image set of $f$. Let the coloring of the top-left arc be $a$ and that of the arc directly below it be $b$. Then the colors of the two bottom arcs of the first branch are $p_{1}b-(p_{1}-1)a$ and $(p_{1}+1)b-p_{1}a$ respectively. Proceeding further, let $c$ denote the coloring of the arc connecting the third and fourth branches. We then obtain the colorings of the two upper arcs of the fourth branch as: $p_{4}c-(p_{4}-1)p_{1}b+(p_{4}-1)(p_{1}-1)a$ and $(p_{4}+1)c-p_{4}p_{1}b+p_{4}(p_{1}-1)a$. Since each arc can only be assigned a single color, we have the following consistency condition:
		\begin{equation}
			p_{4}c-(p_{4}-1)p_{1}b+(p_{4}-1)(p_{1}-1)a\equiv a.
			\label{c1}
		\end{equation}
		Then we have $p_{3}c+p_{3}p_{1}(a-b)-(p_{3}-1)d=d+p_{1}(b-a)$. Let $d$ denote the coloring of the arc connecting the second and third branches. We can then determine the colorings for both the third and second branches. Specifically, the colorings of the two bottom arcs of the third branch are given by:$p_{3}c+p_{3}p_{1}(a-b)-(p_{3}-1)d$ and $(p_{3}+1)c+(p_{3}+1)p_{1}(a-b)-p_{3}d$, thus we have :
		\begin{equation}
			(p_{3}+1)c+(p_{3}+1)p_{1}(a-b)-p_{3}d\equiv c.
			\label{c2}
		\end{equation}
		Furthermore, there is $p_{3}c+p_{3}p_{1}(a-b)-(p_{3}-1)d=d+p_{1}(b-a)$. From the coloring states $b$ and $d$ of the two arcs, we obtain the coloring states of the two bottommost arcs of the second branch as: $p_{2}d-(p_{2}-1)b$ and $(p_{2}+1)d-p_{2}b$. Therefore, we can obtain:
		\begin{equation}
			\begin{cases}
				p_{2}d-(p_{2}-1)b\equiv (p_{1}+1)b-p_{1}a, \\
				(p_{2}+1)d-p_{2}b\equiv d+p_{1}(b-a).
			\end{cases}
			\label{c3}
		\end{equation}
		Thus, the complete coloring scheme of $P_1^4$ can be represented as shown in the following diagram:
		\begin{figure}[H]
			\centering
			\includegraphics[width=1\textwidth,height=0.47\textwidth]{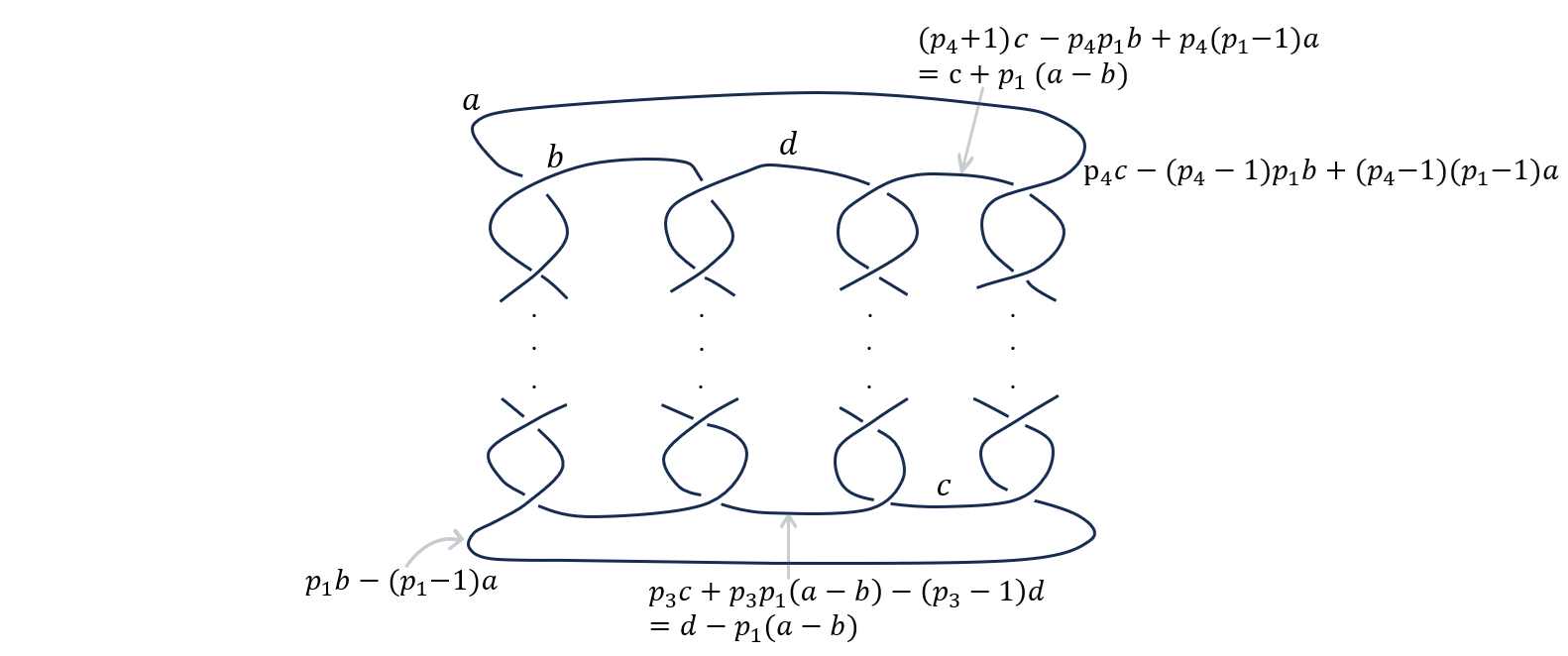}
			\caption{Coloring diagram of $P_1^4$ using the dihedral quandle $\mathbb{Z}_{n}$}
		\end{figure}\par
		In summary, the coloring states of all arcs in $P_1^4$ are completely determined by the variables $a,b,c,d$. Let $f(x_{1})=a,f(x_{2})=b,f(x_{3})=c,f(x_{4})=d,x_{1},x_{2},x_{3},x_{4}\in \mathcal{Q}(P_1^4)$ and $a,b,c,d\in \mathbb{Z}_{n}$. By synthesizing equations (\ref{c1}), (\ref{c2}), and (\ref{c3}), we obtain the following system of congruence equations:
		\begin{equation}
			\begin{cases}
				p_{1}(p_{3}+1)(a-b)+p_{3}(c-d)\equiv 0,\\
				-p_{1}(a-b)+p_{2}(b-d)\equiv 0, \\
				p_{1}(p_{4}-1)(a-b)+p_{4}(c-a)\equiv 0.
			\end{cases}
			\label{el6}
		\end{equation}
		To solve this system of congruence equations, we first perform the following variable substitution:
		\begin{equation}
			\begin{cases}
				a-b=x,\\
				c-b=y,\\
				b-d=z,\\
				d=k.
			\end{cases}
			\nonumber
		\end{equation}
		Expressed in matrix form, the system becomes:
		\begin{equation}
			\left[
			\begin{array}{c}
				x\\
				y\\
				z\\
				k
			\end{array}
			\right]
			=
			\left[
			\begin{array}{cccc}
				1 &-1 & 0 & 0\\
				0 &-1 & 1 & 0\\
				0 & 1 & 0 & -1\\
				0 & 0 & 0 & 1
			\end{array}
			\right]
			\left[
			\begin{array}{c}
				a \\ 
				b \\
				c\\
				d
			\end{array}
			\right]
			=
			B
			\left[
			\begin{array}{c}
				a \\ 
				b \\
				c\\
				d
			\end{array}
			\right].
			\nonumber 
		\end{equation}
		Since $det(B)=-1$ and $\gcd(-1,n)=\gcd(1,n)=1$ the matrix $B$ is invertible over $\mathbb{Z}_{n}$. Consequently, the system of equations (\ref{el6}) is equivalent to the following system:
		\begin{equation}
			\begin{cases}
				p_{1}(p_{3}+1)x+p_{3}(y+z)\equiv 0,\\
				-p_{1}x+p_{2}z\equiv 0,\\
				p_{1}(p_{4}-1)x+p_{4}(y-x)\equiv 0.
			\end{cases}
			\label{el7}
		\end{equation}
		The system of equations (\ref{el7}) can be expressed in matrix form as:
		\begin{equation}
			\left[
			\begin{array}{ccc}
				p_{1}(p_{3}+1)&p_{3}&p_{3}\\
				-p_{1}&0&p_{2}\\
				p_{1}(p_{4}-1)-p_{4}&p_{4}& 0
			\end{array}
			\right]
			\left[
			\begin{array}{c}
				x \\ 
				y \\
				z\\
			\end{array}
			\right]
			=
			C
			\left[
			\begin{array}{c}
				x \\ 
				y \\
				z\\
			\end{array}
			\right]
			\equiv
			\left[
			\begin{array}{c}
				0\\
				0\\
				0
			\end{array}
			\right].
			\nonumber 
		\end{equation}
		It is straightforward to verify that: $det(C)=-(p_{1}p_{2}p_{3}+p_{1}p_{2}p_{4}+p_{1}p_{3}p_{4}+p_{2}p_{3}p_{4})$. We now analyze the rank of matrix $C$:\par 
		\begin{romanlist}
			\item  If the rank of matrix $C$ is 3, i.e., $p_{1}p_{2}p_{3}+p_{1}p_{2}p_{4}+p_{1}p_{3}p_{4}+p_{2}p_{3}p_{4}\not\equiv0$, then the matrix $C$ is invertible. In this case, the homogeneous linear congruence system (\ref{el7}) has only the trivial solution, meaning: $x\equiv y\equiv z\equiv0$. Thus, the solution to the system (\ref{el6}) is $a\equiv b\equiv c\equiv d$, meaning that in this case $|Hom(P_1^4,\mathbb{Z}_{n})|=n$.\par 
			\item  If the rank of matrix $C$ is 0, since $p_{1},p_{2},p_{3},p_{4}\in \mathbb{Z}^{+}$, it must follow that $p_{1}\equiv p_{2}\equiv p_{3}\equiv p_{4}\equiv0$. Thus, the solutions to the system (\ref{el7}) are given by arbitrary $x,y,z\in \mathbb{Z}_{n}$, it immediately follows that the coloring count of $P_1^4$ satisfies $|Hom(P_1^4,\mathbb{Z}_{n})|=n^{4}$.\par 
			\item  If the rank of matrix $C$ is 1, then all $2\times 2$ minors of matrix $C$ vanish modulo $n$. That is
			$p_{1}p_{3}\equiv p_{1}p_{2}p_{3}+p_{1}p_{2}+p_{1}p_{3}\equiv p_{3}p_{4}+p_{1}p_{3}-p_{1}p_{3}p_{4}\equiv p_{1}p_{4}+p_{1}p_{3}+p_{3}p_{4}\equiv p_{1}p_{4}\equiv p_{2}p_{4}-p_{1}p_{2}(p_{4}-1)\equiv p_{2}p_{3}\equiv p_{3}p_{4}\equiv p_{2}p_{4}\equiv 0$.\par 
			Organized as follows:
			\begin{equation}
				p_{1}p_{2}\equiv p_{1}p_{3}\equiv p_{1}p_{4}\equiv p_{2}p_{3}\equiv p_{2}p_{4}\equiv p_{3}p_{4}\equiv 0.
				\label{el8}
			\end{equation}
			Since $n$ is prime and $p_{1},p_{2},p_{3},p_{4}\in \mathbb{Z}^{+}$, equation (\ref{el8}) implies that at least three of the $p_{1},i=1,2,3,4$ must be multiples of $n$. However, because the rank of $C$ is 1, not all four $p_{1},i=1234$ can be divisible by $n$ (which would reduce the rank to 0). Thus, exactly three of the $p_{i},i=1,2,3,4$ are multiples of $n$, while the remaining one is coprime with $n$.\par 
			If $p_{1}$ is coprime with $n$ while $p_{2},p_{3},p_{4}$ are multiples of $n$, then the congruence system (\ref{el6}) reduces to: $p_{1}(a-b)\equiv 0$, which simplifies to: $a-b\equiv 0$; If $p_{2}$ is coprime with $n$ while $p_{1},p_{3},p_{4}$ are multiples of $n$, then the congruence system (\ref{el6}) reduces to: $p_{2}(b-d)\equiv 0$, which simplifies to: $b-d\equiv 0$; If $p_{3}$ is coprime with $n$ while $p_{1},p_{2},p_{4}$ are multiples of $n$, then the congruence system (\ref{el6}) reduces to: $p_{3}(c-d)\equiv 0$, which simplifies to: $c-d\equiv 0$; If $p_{4}$ is coprime with $n$ while $p_{1},p_{2},p_{3}$ are multiples of $n$, then the congruence system (\ref{el6}) reduces to: $p_{4}(c-a)\equiv 0$, which simplifies to: $c-a\equiv 0$. Summary: When exactly one of the parameters $p_{i},i=1,2,3,4$ is coprime to $n$ (while the other three are multiples of $n$), the coloring count of $P_1^4$ satisfies: $|Hom(P_1^4,\mathbb{Z}_{n})|=n^3$.\par 
			\item  If the rank of matrix $C$ is 2, then at least one of the pairwise products $p_{1}p_{2},p_{1}p_{3},p_{1}p_{4},p_{2}p_{3},p_{2}p_{4},p_{3}p_{4}$ is non-zero modulo $n$, while the cubic sum $h=p_{1}p_{2}p_{3}+p_{1}p_{2}p_{4}+p_{1}p_{3}p_{4}+p_{2}p_{3}p_{4}\equiv 0$ vanishes.\par 
			Without loss of generality, if $p_{1}p_{2}\not \equiv 0$ where $n$ is prime, then both $p_{1}$ and $p_{2}$ must be coprime with $n$. Clearly, neither the case where $p_{3}\not \equiv 0,p_{4}\equiv 0$ nor $p_{3}\equiv 0,p_{4}\not \equiv 0$ can satisfy $h\equiv 0$. When $p_{3}\equiv p_{4}\equiv 0$ the congruence system (\ref{el7}) reduces to:
			\begin{equation}
				\begin{cases}
					p_{1}x\equiv 0,\\
					-p_{1}x+p_{2}x\equiv 0.
				\end{cases}
				\label{3.19}
			\end{equation}
			It is straightforward to verify that the solutions to the congruence system (\ref{3.19}) are given by: $x\equiv z\equiv 0,y\in \mathbb{Z}_{n}$. Consequently, in this case, the solutions to the congruence system (\ref{el6}) satisfy:
			\begin{equation}
				\begin{cases}
					a=d=b,\\
					c=b+k\qquad k=0,1,...,n-1.
				\end{cases}
				\nonumber
			\end{equation}
			Therefore, at this moment, $|Hom(P_1^4,\mathbb{Z}_{n})|=n^{2}$.\par 
			When $p_{3}\not \equiv 0,p_{4}\not \equiv 0$, since n is prime and $p_{1}p_{2}\not \equiv 0$, it follows that all $p_{1},p_{2},p_{3},p_{4}$ are units in $\mathbb{Z}_{n}$. That is, there exist multiplicative inverses $p_{i}^{-1}\in \mathbb{Z}_{n}$ satisfying $p_{i}^{-1}p_{i}\equiv 1\ (mod\ n)$, for each $i=1,2,3,4$. Furthermore, given that the matrix $C$ has rank 2, the solution set of the congruence system (\ref{el7}) is characterized by:
			\begin{equation}
				\begin{cases}
					z\equiv p_{2}^{-1}p_{1}x,\\
					y\equiv p_{4}^{-1}[p_{4}-p_{1}(p_{4}-1)]x.
				\end{cases}
			\end{equation}
			Consequently, the solution to congruence system (\ref{el6}) in this case is given by:
			\begin{equation}
				\begin{cases}
					d\equiv b-p_{2}^{-1}p_{1}(a-b),\\
					c\equiv b+p_{4}^{-1}[p_{4}-p_{1}(p_{4}-1)](a-b).
				\end{cases}
			\end{equation}
			Thus, in this case we also have: $|Hom(P_1^4,\mathbb{Z}_{n})|=n^{2}$.\par 
			In summary, when $h\equiv 0$, and at least one of  $p_{1}p_{2},p_{1}p_{3},p_{1}p_{4},p_{2}p_{3},p_{2}p_{4},p_{3}p_{4}$ is non-zero modulo $n$, we conclude that $|Hom(P_1^4,\mathbb{Z}_{n})|=n^{2}$.
		\end{romanlist}
	\end{proof}
	\begin{example}
		Given an oriented pretzel link $(p_{1},p_{2},p_{3},p_{4})=(11,12,15,15)$ and a dihedral quandle $\mathbb{Z}_{n}$, when $n=7$, we have $h=p_{1}p_{2}p_{3}+p_{1}p_{2}p_{4}+p_{1}p_{3}p_{4}+p_{2}p_{3}p_{4}=0\ (mod\ 7)$. Substituting the values of $p_{1},p_{2},p_{3},p_{4}$ into the congruence system (\ref{el7}) and simplifying, we obtain:
		\begin{equation}
			\begin{cases}
				x+y+z\equiv 0,\\
				3x+5z\equiv 0,\\
				6x+y\equiv 0.
			\end{cases}
			\label{ex14}
		\end{equation}
		Solving the congruence system (\ref{ex14}) yields:
		\begin{equation}
			\begin{cases}
				x\equiv 3z,\\
				y\equiv 3z.
			\end{cases}
			\nonumber
		\end{equation}
		That is, at this time, the solution to the congruence system (\ref{el6}) is:
		\begin{equation}
			\begin{cases}
				a\equiv 4b+4d,\\
				c\equiv 4b+4d.
			\end{cases}
			\nonumber
		\end{equation}
		Thus, the coloring number of the pretzel link (11,12,15,15) at this time is $n^{2}=49$.\par 
		It is also clear that $p_{1}$ and $p_{2}$ are both coprime with $n$, so $p_{1}p_{2}\not \equiv 0\ (mod\ n)$ which is consistent with the conclusion of Theorem \ref{th4p}.
	\end{example}
	\begin{corollary}
		Given an oriented pretzel link $P_2^4=(p_{1},p_{2},p_{3},-p_{4})$ and a dihedral quandle $\mathbb{Z}_{n}$, where $n$ is prime, let $h=p_{1}p_{2}p_{3}-p_{1}p_{2}p_{4}-p_{1}p_{3}p_{4}-p_{2}p_{3}p_{4}$. The number of colorings of $P_2^4$ is:
		\begin{equation}
			|Hom(P_2^4,\mathbb{Z}_n)|=
			\begin{cases}
				n\qquad \ h\not\equiv 0;\\
				n^{2}\qquad h\equiv 0,\ and\ at\ least\ one\ p_{i}p_{j}\not \equiv 0\ hold\ for\ i\neq j,
				\\ \qquad \quad i,\ j =1,2,3,4;\\
				n^{3}\qquad  only\ one\ of\ p_{i},i=1,2,3,4\ is\ coprime\ with\ n;\\
				n^{4}\qquad p_{1}\equiv p_{2}\equiv p_{3}\equiv p_4\equiv 0.
			\end{cases}
			\nonumber
		\end{equation}
	\end{corollary}
	\begin{proof}
		For any given pretzel link $P_2^4=(p_{1},p_{2},p_{3},-p_{4})$ and dihedral quandle $\mathbb{Z}_{n}$ where $n$ is prime, analogous to Theorem \ref{th4p}, let the fundamental quandle of $P_2^4$ be $\mathcal{Q}(P_2^4)$. For any $f\in Hom(\mathcal{Q}(P_2^4),\mathbb{Z}_{n})$, the image of $f$ is determined by the coloring of four arcs, denoted $a,b,c,d$. Let $f(x_{1})=a,f(x_{2})=b,f(x_{3})=c,f(x_{4})=d,x_{1},x_{2},x_{3},x_{4}\in \mathcal{Q}(P_2^4)$ and $a,b,c,d\in \mathbb{Z}_{n}$. The coloring of all arcs of $P_2^4$ can be shown in the following figure:
		\begin{figure}[H]
			\centering
			\includegraphics[width=1\textwidth,height=0.47\textwidth]{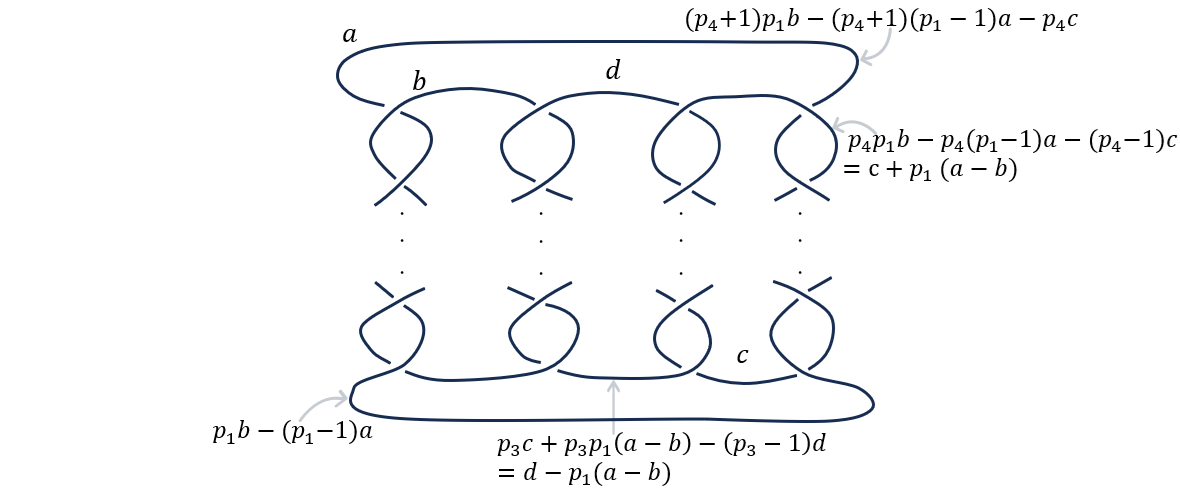}
			\caption{Coloring diagram of $P_2^4$ using the dihedral quandle $\mathbb{Z}_{n}$}
		\end{figure}\par 
		The coloring system for $P_2^4$ is:
		\begin{equation}
			\begin{cases}
				p_{1}(p_{3}+1)(a-b)+p_{3}(c-d)\equiv 0,\\
				-p_{1}(a-b)+p_{2}(b-d)\equiv 0,\\
				p_{1}(p_{4}+1)(a-b)+p_{4}(c-a)\equiv 0.
			\end{cases}
			\label{eq-4}
		\end{equation}
		Solve the system of congruences (\ref{eq-4}), and similarly to Theorem \ref{th4p}, we can obtain that the coloring number of $P_2^4$ is:
		\begin{equation}
			|Hom(P_2^4,\mathbb{Z}_n)|=
			\begin{cases}
				n\qquad \ h\not\equiv 0;\\
				n^{2}\qquad h\equiv 0,\ and\ at\ least\ one\ p_{i}p_{j}\not \equiv 0\ hold\ for\ i\neq j,
				\\ \qquad \quad i,\ j =1,2,3,4;\\
				n^{3}\qquad  only\ one\ of\ p_{i},i=1,2,3,4\ is\ coprime\ with\ n;\\
				n^{4}\qquad p_{1}\equiv p_{2}\equiv p_{3}\equiv 0.
			\end{cases}
			\nonumber
		\end{equation}
	\end{proof}
	\begin{corollary}
		Given an oriented pretzel link $P_3^4=(p_{1},p_{2},-p_{3},-p_{4})$ and a dihedral quandle $\mathbb{Z}_{n}$, where n is prime. Let $h=-p_{1}p_{2}p_{3}-p_{1}p_{2}p_{4}+p_{1}p_{3}p_{4}+p_{2}p_{3}p_{4}$, the number of colorings of $P_3^4$ is:
		\begin{equation}
			|Hom(P_3^4,\mathbb{Z}_n)|=
			\begin{cases}
				n\qquad \ h\not\equiv 0;\\
				n^{2}\qquad h\equiv 0,\ and\ at\ least\ one\ p_{i}p_{j}\not \equiv 0\ hold\ for\ i\neq j,
				\\ \qquad \quad i,\ j =1,2,3,4;\\
				n^{3}\qquad  only\ one\ of\ p_{i},i=1,2,3,4\ is\ coprime\ with\ n;\\
				n^{4}\qquad p_{1}\equiv p_{2}\equiv p_{3}\equiv p_4\equiv 0.
			\end{cases}
			\nonumber
		\end{equation}
	\end{corollary}
	\begin{proof}
		For any given pretzel link $P_3^4=(p_{1},p_{2},-p_{3},-p_{4})$ and dihedral quandle $\mathbb{Z}_{n}$ where $n$ is prime, analogous to Theorem \ref{th4p}, let the fundamental quandle of $P_3^4$ be $\mathcal{Q}(P_3^4)$. For any $f\in Hom(\mathcal{Q}(P_3^4),\mathbb{Z}_{n})$, the image of $f$ is determined by the coloring of four arcs, denoted $a,b,c,d$. Let $f(x_{1})=a,f(x_{2})=b,f(x_{3})=c,f(x_{4})=d,x_{1},x_{2},x_{3},x_{4}\in \mathcal{Q}(P_3^4)$ and $a,b,c,d\in \mathbb{Z}_{n}$. The coloring of all arcs of $P_3^4$ can be shown in Fig.\ref{t12}:
		\begin{figure}[!ht]
			\centering
			\includegraphics[width=1\textwidth,height=0.47\textwidth]{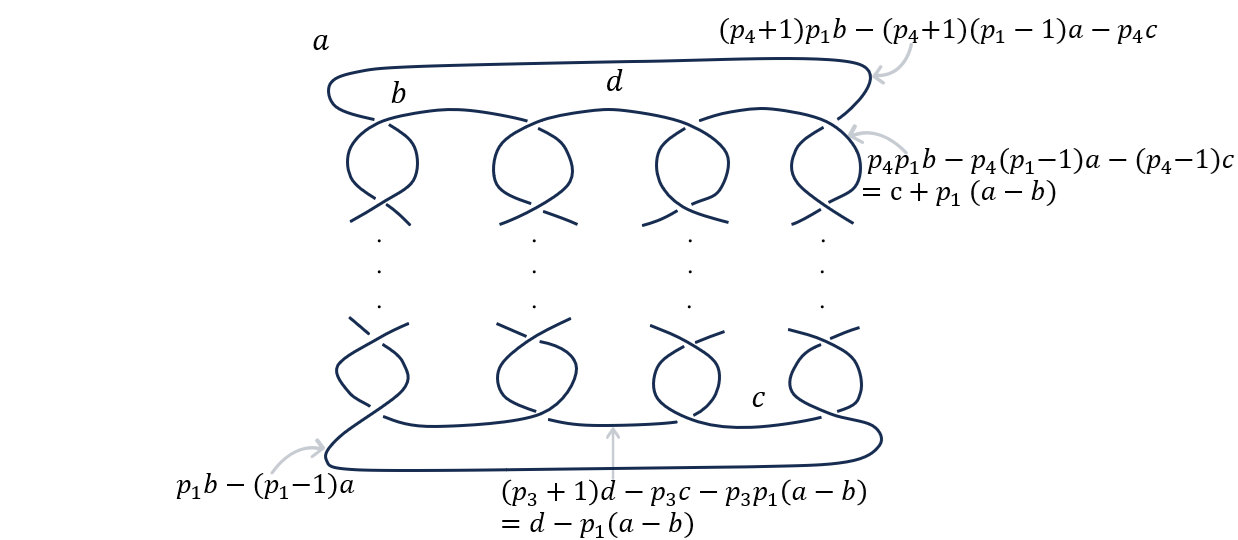}
			\caption{Coloring diagram of $P_3^4$ using the dihedral quandle $\mathbb{Z}_{n}$}
			\label{t12}
		\end{figure}\par 
		The coloring system of $P_3^4$ is:
		\begin{equation}
			\begin{cases}
				p_{1}(p_{4}+1)(a-b)+p_{4}(c-a)\equiv 0,\\
				p_{1}(a-b)+p_{2}(d-b)\equiv 0,\\
				p_{1}(p_{3}-1)(a-b)+p_{3}(c-d)\equiv 0.
			\end{cases}
			\label{eq-3-4}
		\end{equation}
		Solve the system of congruences (\ref{eq-3-4}), and similarly to Theorem \ref{th4p}, we can obtain that the coloring number of $P_3^4$ is:
		\begin{equation}
			|Hom(P_3^4,\mathbb{Z}_n)|=
			\begin{cases}
				n\qquad \ h\not\equiv 0;\\
				n^{2}\qquad h\equiv 0,\ and\ at\ least\ one\ p_{i}p_{j}\not \equiv 0\ hold\ for\ i\neq j,
				\\ \qquad \quad i,\ j =1,2,3,4;\\
				n^{3}\qquad  only\ one\ of\ p_{i},i=1,2,3,4\ is\ coprime\ with\ n;\\
				n^{4}\qquad p_{1}\equiv p_{2}\equiv p_{3}\equiv p_4\equiv0.
			\end{cases}
			\nonumber
		\end{equation}
	\end{proof}
	\begin{corollary}
		Given an oriented pretzel link $P_4^4=(p_{1},-p_{2},p_{3},-p_{4})$ and a dihedral quandle $\mathbb{Z}_{n}$, where n is prime, let $h=-p_{1}p_{2}p_{3}+p_{1}p_{2}p_{4}-p_{1}p_{3}p_{4}+p_{2}p_{3}p_{4}$. The number of colorings of $P_4^4$ is:
		\begin{equation}
			|Hom(P_4^4,\mathbb{Z}_n)|=
			\begin{cases}
				n\qquad \ h\not\equiv 0;\\
				n^{2}\qquad h\equiv 0,\ and\ at\ least\ one\ p_{i}p_{j}\not \equiv 0\ hold\ for\ i\neq j,
				\\ \qquad \quad i,\ j =1,2,3,4;\\
				n^{3}\qquad  only\ one\ of\ p_{i},i=1,2,3,4\ is\ coprime\ with\ n;\\
				n^{4}\qquad p_{1}\equiv p_{2}\equiv p_{3}\equiv p_4\equiv 0.
			\end{cases}
			\nonumber
		\end{equation}
	\end{corollary}
	\begin{proof}
		For any given pretzel link $P_4^4=(p_{1},-p_{2},p_{3},-p_{4})$ and dihedral quandle $\mathbb{Z}_{n}$ where $n$ is prime, analogous to Theorem \ref{th4p}, let the fundamental quandle of $P_4^4$ be $\mathcal{Q}(P_4^4)$. For any $f\in Hom(\mathcal{Q}(P_4^4),\mathbb{Z}_{n})$, the image of $f$ is determined by the coloring of four arcs, denoted $a,b,c,d$. Let $f(x_{1})=a,f(x_{2})=b,f(x_{3})=c,f(x_{4})=d,x_{1},x_{2},x_{3},x_{4}\in \mathcal{Q}(P_4^4)$ and $a,b,c,d\in \mathbb{Z}_{n}$. The coloring of all arcs of $P_4^4$ can be shown in the following figure:\par 
		\begin{figure}[H]
			\centering
			\includegraphics[width=1\textwidth,height=0.47\textwidth]{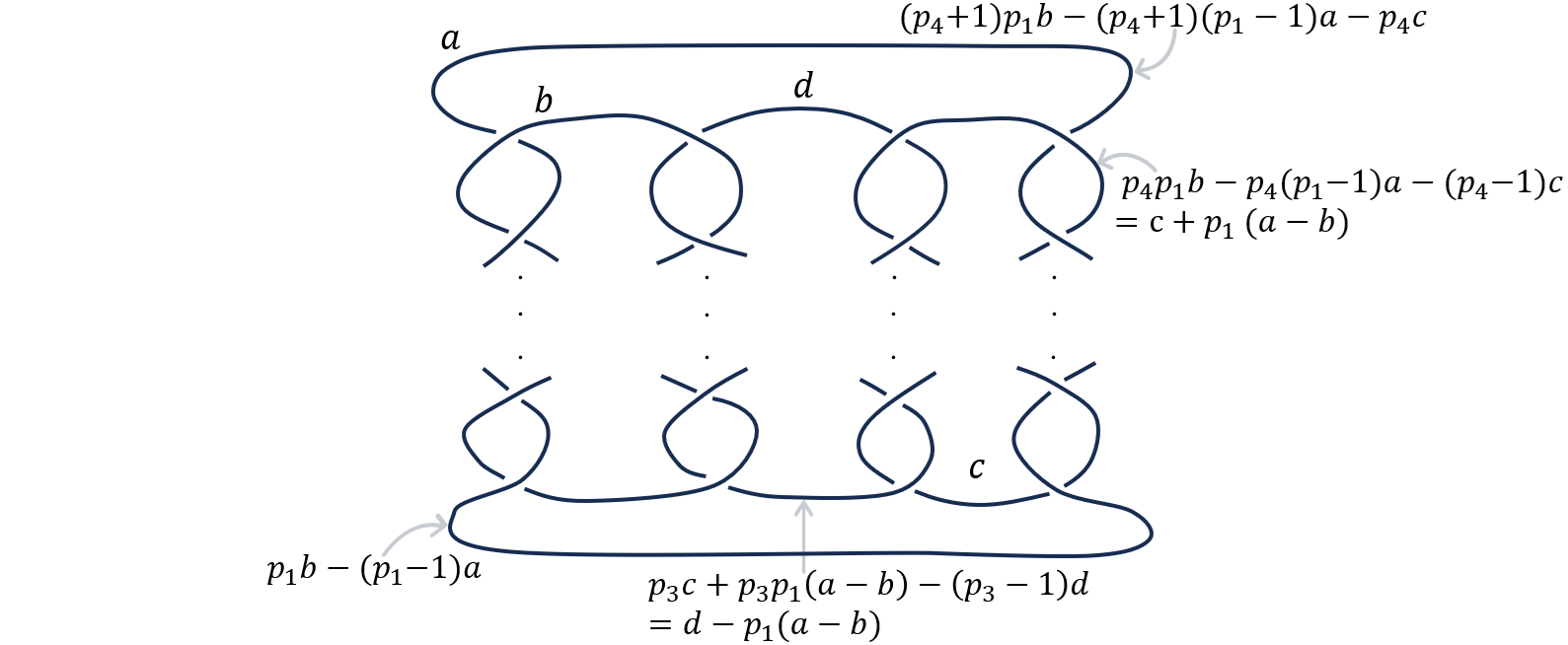}
			\caption{Coloring diagram of $P_4^4$ using the dihedral quandle $\mathbb{Z}_{n}$}
			\label{tap13}\par 
		\end{figure}
		The coloring system for $P_4^4$ is:
		\begin{equation}
			\begin{cases}
				p_{1}(p_{3}+1)(a-b)+p_{3}(c-d)\equiv 0,\\
				p_{2}(d-b)+p_{1}(b-a)\equiv 0,\\
				p_{1}(p_{4}+1)(a-b)+p_{4}(c-a)\equiv 0.
			\end{cases}
			\label{3.27}
		\end{equation}
		Solve the system of congruences (\ref{3.27}), and similarly to Theorem \ref{th4p}, we can obtain that the coloring number of $P_4^4$ is:
		\begin{equation}
			|Hom(P_4^4,\mathbb{Z}_n)|=
			\begin{cases}
				n\qquad \ h\not\equiv 0;\\
				n^{2}\qquad h\equiv 0,\ and\ at\ least\ one\ p_{i}p_{j}\not \equiv 0\ hold\ for\ i\neq j,
				\\ \qquad \quad i,\ j =1,2,3,4;\\
				n^{3}\qquad  only\ one\ of\ p_{i},i=1,2,3,4\ is\ coprime\ with\ n;\\
				n^{4}\qquad p_{1}\equiv p_{2}\equiv p_{3}\equiv p_4\equiv 0.
			\end{cases}
			\nonumber
		\end{equation}
	\end{proof}
	From the above analysis, it can be seen that when $n$ is a prime number, the coloring of different types of pretzel links is partially dependent on the value of $h$. Moreover, for pretzel links $(q_{1},q_{2},q_{3},q_{4})$ that do not consist entirely of positive crossings (i.e., where some $q_{i}<0$), the value of $h$ for $(q_{1},q_{2},q_{3},q_{4})$ can be obtained by simply substituting $p_{i}$ in Theorem \ref{th4p} for $q_{i}$.
	\subsection{Quandle coloring spaces of General pretzel links}
	From Theorem \ref{th2.13} and Theorem \ref{th2.14}, we can deduce that there are $\frac{1}{2m}(\sum_{d|m}\phi(d)2^{m/d}+\sum_{d|m,2|d}\phi(d)2^{m/d})$ different $m$-pretzel links. (The number of distinct $m$-pretzel links is equivalent to the number of orbits of the sequence $(s_1, s_2, \dots, s_m)$, where each $s_i = \pm 1$, under the group action generated by cyclic shifts and global inversion.) Based on the study of the coloring number of 4-pretzel links in Section \ref{sec3.2}, we can, without loss of generality, focus only on the coloring number of the first type of n-pretzel links here.
	\begin{theorem}\label{thmp}
		Given an oriented pretzel link $P^m=(p_{1},...,p_{m})$ and a dihedral quandle $\mathbb{Z}_{n}$, where n is prime , $m,p_{1},...,p_{m}\in \mathbb{Z}^{+}$. Let $h=\sum p_{i_{1}}p_{i_{2}}...p_{i_{m-1}},1\leq i_{1}<i_{2}<...<i_{m-1}\leq m$. The coloring number of $P^m$ is
		\begin{equation}
			|Hom(P^m,\mathbb{Z}_{n})|=
			\begin{cases}
				n\qquad \;\  h\not \equiv 0;\\
				n^{2}\qquad h\equiv 0\ and\ p_{1}p_{2}...p_{m-2}\equiv ...\equiv p_{3}p_{4}...p_{m}\equiv 0\ do\ not\\ \qquad \quad \, all\ hold;\\
				\,\vdots \\
				n^{m-2}\quad \ h\equiv 0\ and\ among\ p_{1},p_{2},...,p_{m}\ two\ are\ coprime\\ \qquad \quad \ \ with\  n\ while\ the\ rest\ are\ multiples\ of\ n;\\
				n^{m-1}\quad \  among\ p_{1},p_{2},...,p_{m}\ m-1\ are\ multiples\ of\ n\\ \qquad \quad \ \ \,  and\ only\ one\ is\ coprime\ with\ n;\\
				n^{m}\qquad \;p_{1}\equiv p_{2}\equiv ...\equiv p_{m}\equiv 0.
			\end{cases}
			\nonumber
		\end{equation}
	\end{theorem}
	\begin{proof}
		For any given pretzel link $P^m=(p_{1},p_{2},...,p_{m})$ and dihedral quandle $\mathbb{Z}_{n}$ where $n$ is prime. We denote the fundamental quandle of $P^m$ as $\mathcal{Q}(P^m)$. Then, for any $f\in Hom(\mathcal{Q}(P^m),\mathbb{Z}_{n})$, similar to Theorem \ref{th4p} we first examine the image set of $f$, that is, coloring the entire $P^m$ with the dihedral quandle, which yields the coloring diagram for each arc:\par 
		\begin{figure}[H]
			\centering
			\includegraphics[width=1\linewidth]{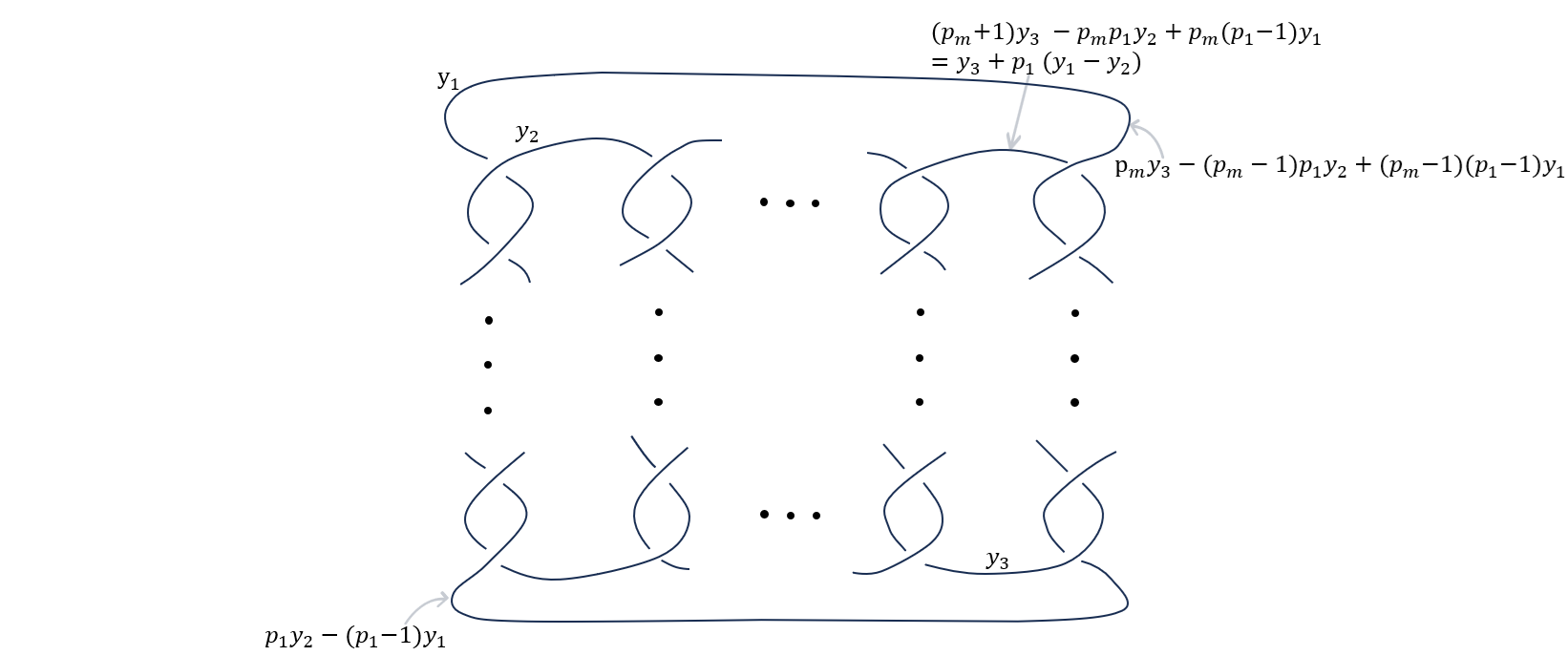}
			\caption{Coloring diagram of $P^m$ using the dihedral quandle $\mathbb{Z}_{n}$}
			\label{tnp}
		\end{figure}\par
		Furthermore, we can see that the image of $f$ is determined solely by the colorings of the $m$ arcs, meaning the image of $f$ depends on the values of $y_{1},y_{2},...,y_{m}$, while the colorings of the remaining arcs are linear combinations of them. Let $f(x_{1})=y_{1},f(x_{2})=y_{2},...,f(x_{m})=y_{m}$, where $x_{1},x_{2},...,x_{m}\in \mathcal{Q}(P^m)$, and $y_{1},y_{2},...,y_{m}\in \mathbb{Z}_{n}$. Then, when $m$ is an odd number, we obtain the coloring equation system for $P^m$:
		\begin{equation}
			\begin{cases}
				p_{1}(p_{m}-1)(y_{1}-y_{2})+p_{m}(y_{3}-y_{1})\equiv 0,\\
				p_{1}(p_{m-1}+1)(y_{1}-y_{2})+p_{m-1}(y_{3}-y_{4})\equiv 0,\\
				p_{1}(m_{m-2}-1)(y_{1}-y_{2})+p_{m-2}(y_{5}-y_{4})\equiv 0,\\
				\qquad \cdots \\
				p_{1}(p_{3}-1)(y_{1}-y_{2})+p_{3}(y_{m}-y_{m-1})\equiv 0,\\
				p_{1}(p_{2}+1)(y_{1}-y_{2})+p_{2}(y_{m}-y_{2})\equiv 0.
			\end{cases}
			\label{np1}
		\end{equation}
		When $m$ is an even number, we obtain the coloring equation system for $P^m$:
		\begin{equation}
			\begin{cases}
				p_{1}(p_{m}-1)(y_{1}-y_{2})+p_{m}(y_{3}-y_{1})\equiv 0,\\
				p_{1}(p_{m-1}+1)(y_{1}-y_{2})+p_{m-1}(y_{3}-y_{4})\equiv 0,\\
				p_{1}(m_{m-2}-1)(y_{1}-y_{2})+p_{m-2}(y_{5}-y_{4})\equiv 0,\\
				\qquad \cdots \\
				p_{1}(p_{3}-1)(y_{1}-y_{2})+p_{3}(y_{m}-y_{m-1})\equiv 0,\\
				p_{1}(y_{1}-y_{2})+p_{2}(y_{m}-y_{2})\equiv 0.
			\end{cases}
			\label{el31}
		\end{equation}
		Let us first solve the system of congruences (\ref{el31}). Given that it consists of $m$ equations in $m-1$ unknowns, we begin by making a change of variables:
		\begin{equation}
			\begin{cases}
				y_{1}-y_{2}=z_{1},\\
				y_{3}-y_{2}=z_{2},\\
				\qquad\cdots \\
				y_{m}-y_{2}=z_{m-1},\\
				y_{m}=z_{m}.
			\end{cases}
			\nonumber
		\end{equation}
		In matrix form, this can be expressed as:
		\begin{equation}
			\left[
			\begin{array}{c}
				z_{1}\\
				z_{2} \\
				\vdots \\
				z_{m-1}\\
				z_{m}
			\end{array}
			\right]
			=
			\begin{bmatrix}
				1&-1&\cdots &0&0\\
				0&-1&\cdots &0&0\\
				\vdots &\vdots &\ddots &\vdots &\vdots \\
				0&-1&\cdots &0&1\\
				0&0 &\cdots &0&1
			\end{bmatrix}
			\left[
			\begin{array}{c}
				y_{1} \\ 
				y_{2} \\
				\vdots \\
				y_{m-1}\\
				y_{m}
			\end{array}
			\right]
			=
			D
			\left[
			\begin{array}{c}
				y_{1}\\ 
				y_{2}\\
				\vdots \\
				y_{m-1}\\
				y_{m}
			\end{array}
			\right]
			\nonumber 
		\end{equation}
		Since $det(D)=1$, the matrix $D$ is invertible. Consequently, the system of congruences (\ref{el31}) is equivalent to the following system (\ref{el32}):
		\begin{equation}
			\begin{cases}
				p_{1}(p_{m}-1)z_{1}+p_{m}(z_{2}-z_{1})\equiv 0,\\
				p_{1}(p_{m-1}+1)z_{1}+p_{m-1}(z_{2}-z_{3})\equiv 0,\\
				p_{1}(p_{m-2}-1)z_{1}+p_{m-2}(z_{4}-z_{3})\equiv 0,\\
				\qquad \cdots \\
				p_{1}(p_{3}-1)z_{1}+p_{3}(z_{m-1}-z_{m-2})\equiv 0,\\
				p_{1}(p_{2}+1)z_{1}+p_{2}z_{m-1}\equiv 0.
			\end{cases}
			\label{el32}
		\end{equation}
		The system of congruences (\ref{el32}) can be represented in matrix form as:
		\begin{equation}
			\begin{bmatrix}
				p_{1}(p_{m}-1)-p_{m}&p_{m}&\cdots &0 &0\\
				p_{1}(p_{m-1}+1)&p_{m-1}&\cdots &0 &0\\
				\vdots &\vdots &\ddots &\vdots &\vdots \\
				p_{1}(p_{3}-1)&0&\cdots &-p_{3}&p_{3}\\
				p_{1}(p_{2}+1)&0&\cdots &0&p_{2}
			\end{bmatrix}
			\left[
			\begin{array}{c}
				z_{1} \\ 
				z_{2} \\
				\vdots \\
				z_{m-1}
			\end{array}
			\right]
			=
			G
			\left[
			\begin{array}{c}
				z_{1} \\ 
				z_{2} \\
				\vdots \\
				z_{m-1}
			\end{array}
			\right]
			\equiv
			\left[
			\begin{array}{c}
				0\\
				0\\
				\vdots \\
				0
			\end{array}
			\right]
			\label{elm}
		\end{equation}
		\par 
		We proceed to analyze the rank of matrix $E$:\par 
		\begin{romanlist}[(ii)]
			\item If $det(G)\not \equiv 0$, hen $G$ has full rank, and consequently the homogeneous system of congruences (\ref{el32}) admits only the trivial solution. In this case, we have $|Hom(P^m,\mathbb{Z}_{n})|=n$.\par 
			\item If the rank of $G$ is 0, then every element of $G$ is a multiple of $n$, i.e., $p_{1}\equiv p_{2}\equiv ...p_{m}\equiv 0$. In this case, the solutions to the system of congruences (\ref{el32}) are arbitrary elements $z_{1},z_{2},...,z_{m}\in \mathbb{Z}_{n}$. It follows immediately that $|Hom(P^m,\mathbb{Z}_{n})|=n^{m}$.\par 
			\item If the rank of $G$ is 1, then the determinant of every $2\times 2$ submatrix of $G$ is a multiple of $n$, i.e.,
			\begin{equation}
				p_{i}p_{j}\equiv 0\qquad (there\ are\ C_{m}^{2}\ distinct\ pairs\ of\ indices\ (i,j)\ where i<j).
				\label{ld2}
			\end{equation}
			Since $rank(G)=1$, at least one of $p_{1},p_{2},...,p_{m}$ is not a multiple of $n$. Without loss of generality, assume $p_{1}\not \equiv 0$. Then, from (\ref{ld2}), it follows that $p_{2}\equiv p_{3}\equiv ...p_{m}\equiv 0$. In this case, the solution to the system of equations (\ref{el32}) is:
			\begin{equation}
				\begin{cases}
					z_{1}\equiv 0,\\
					z_{2},...,z_{m}\in \mathbb{Z}_{n}.
				\end{cases}
				\nonumber
			\end{equation}
			It is straightforward to see that in this case, we have $|Hom(P^m,\mathbb{Z}_{n})|=n^{m-1}$. Furthermore, we can conclude that when $rank(G)=1$, among $p_{1},p_{2},...,p_{m}$ exactly one is coprime with $n$, while all the others are multiples of $n$.
			\item If the rank of $G$ is 2, then at least one equation in (\ref{ld2}) does not hold, and all third-order minors of matrix $G$ have determinants that are multiples of $n$, i.e.,
			\begin{equation}
				p_{i}p_{j}p_{k}\equiv 0\quad (there\ are\ C_{m}^{3}\ distinct\ pairs\ of\ indices\ (i,j,k)\ where\ i<j<k).
				\label{ld3}
			\end{equation}
			Without loss of generality, we assume $p_{1}p_{2}\not \equiv 0$. Then, we have $p_{3}\equiv p_{4}\equiv ...p_{m}\equiv 0$ and the solution to the congruence system (\ref{el32}) is:
			\begin{equation}
				\begin{cases}
					z_{1}\equiv 0,\\
					z_{m-1}\equiv 0,\\
					z_{2},...,z_{m-2}\in \mathbb{Z}_{n}.
				\end{cases}
				\nonumber
			\end{equation}
			Thus, in this case we have $|Hom(P^m,\mathbb{Z}_{n})|=n^{m-2}$, and among $p_{1},p_{2},...,p_{m}$, exactly two are coprime with n, while the remaining ones are multiples of $n$.\par 
			\item  By analogy, we assume that the conclusion holds when the rank of $G$ is m-3.\par 
			\item When the rank of $G$ is $m-2$, we have $det(G)\equiv 0$, and at least one of the $m-2$-th order minors of $G$ is not divisible by n. That is:
			\begin{equation}
				p_{1}p_{2}...p_{m-2}\equiv ...\equiv p_{3}p_{4}...p_{m}\equiv 0.
				\label{ld4}
			\end{equation}
			The congruence relations (\ref{ld4}) do not all hold simultaneously. Without loss of generality, assume $p_{1}p_{2}...p_{m-2}\not \equiv 0$, then it follows that each $p_{1},...p_{m-2}$ is coprime with $n$.\par 
			\begin{romanlist}[(b)]
				\item  When $p_{m-1}\not \equiv 0$ and $p_{m}\equiv 0$, we have $det(G)=p_{1}...p_{m-1}\not \equiv 0$ which contradicts the assumption. Similarly, the case $p_{m-1}\equiv 0$ and $p_{m}\not \equiv 0$ is also impossible.\par 
				\item If $p_{m-1}\equiv p_{m}\equiv 0$, then the system of congruences (\ref{el31}) is equivalent to:
				\begin{equation}
					\begin{cases}
						y_{1}\equiv y_{2},\\
						p_{1}(m_{m-2}-1)(y_{1}-y_{2})+p_{m-2}(y_{5}-y_{4})\equiv 0,\\
						\qquad \cdots \\
						p_{1}(p_{3}-1)(y_{1}-y_{2})+p_{3}(y_{m}-y_{m-1})\equiv 0,\\
						p_{1}(p_{2}+1)(y_{1}-y_{2})+p_{2}(y_{m}-y_{2})\equiv 0.
					\end{cases}
					\nonumber
				\end{equation}
				Since $p_{1},...p_{m-2}$ are all coprime with $n$, it is straightforward to verify that the matrix $H=\begin{bmatrix}
					-p_{m-2}&p_{m-2}&\cdots &0&0\\
					\vdots &\vdots& \ddots &\vdots&\vdots\\
					0&0&\cdots &-p_{3}&p_{3}\\
					0&0&\cdots &0&p_{2}
				\end{bmatrix}$ is invertible. That is, there exists an invertible matrix $H^{-1}$ over $\mathbb{Z}_{n}$ such that $HH^{-1}=E$ (the identity matrix). Consequently, we have:
				\begin{equation}
					\left[
					\begin{array}{c}
						z_{3}\\
						\vdots\\
						z_{m-1}
					\end{array}
					\right]
					=H^{-1}
					\begin{bmatrix}
						p_{1}(p_{m-2}-1)&0&\cdots &0\\
						0&p_{1}(p_{m-3}+1)&\cdots &0\\
						\vdots &\vdots&\ddots &\vdots\\
						0&0&\cdots &p_{1}(p_{2}+1)
					\end{bmatrix}
					\left[
					\begin{array}{c}
						z_{1}\\
						z_{1}\\
						\vdots\\
						z_{1}
					\end{array}
					\right]
					\nonumber
				\end{equation}
				Therefore, we conclude that in this case $|Hom(P^m,\mathbb{Z}_{n})|=n^2$.\par 
				\item  When $p_{m-1}\not \equiv 0$ and $p_{m}\not \equiv 0$, all $p_{1},...p_{m}$ become invertible elements in $\mathbb{Z}_{n}$. Thus, the matrix $I=\begin{bmatrix}
					p_{m-1}&-p_{m-1}&\cdots &0&0\\
					0&-p_{m-2}&\cdots &0&0\\
					\vdots &\vdots &\ddots &\vdots &\vdots \\
					0&0&\cdots &-p_{3}&p_{3}\\
					0&0&\cdots &0&p_{2}
				\end{bmatrix}$ is invertible. Furthermore, since the rank of matrix $G$ is $m-2$ in this case, the solution to the system of congruences (\ref{el32}) is given by:
				\begin{equation}
					\left[
					\begin{array}{c}
						z_{2}\\
						z_{3}\\
						\vdots \\
						z_{m-1}
					\end{array}
					\right]
					=I^{-1}
					\begin{bmatrix}
						p_{1}(p_{m-1}+1)&0&\cdots &0\\
						0&p_{1}(p_{m-2}-1)&\cdots &0\\
						\vdots &\vdots&\ddots &\vdots\\
						0&0&\cdots &p_{1}(p_{2}+1)
					\end{bmatrix}
					\left[
					\begin{array}{c}
						z_{1}\\
						z_{1}\\
						\vdots\\
						z_{1}
					\end{array}
					\right]
					\nonumber
				\end{equation} 
				Thus, in this case, the solution to the congruence system (\ref{el31}) has two free variables $y_{1}$ and $y_{2}$, which leads to $|Hom(P^m,\mathbb{Z}_{n})|=n^2$.\par 
			\end{romanlist}
		\end{romanlist}
		\par 
		Similarly to the discussion of the system of congruence equations (\ref{el31}), we can conclude that the result of the system of congruence equations (\ref{np1}) is identical to that of (\ref{el31}).\par 
		This completes the proof.
	\end{proof}
	\eject
	
	\noindent
	
	\section{The quandle coloring quivers of pretzel links}
	
		\begin{theorem}\label{th4.1}
		Given an oriented pretzel link $P_1^3=(p_{1},p_{2},p_{3})$ and a dihedral quandle $\mathbb{Z}_{n}$, where $p_{1},p_{2},p_{3},n\in \mathbb{Z}^{+}$. Let $P=p_1p_2+p_1p_3+p_2p_3,Q=\gcd(p_{1},p_{2},p_3),d_{1}=\gcd(Q,n),d_2=\gcd(\frac{|P|}{Q},n)$. If $d_1,d_2$ are primes, then the full quandle coloring quiver of $P_1^3$ is\\ $$\mathcal{Q}_{\mathbb{Z}_{n}}(P_1^3)=(G_{1}\overleftarrow{\nabla}_{\widehat{\frac{n}{d_{1}}}}G_{2})\bigsqcup (G_{1}\overleftarrow{\nabla}_{\widehat{\frac{n}{d_{2}}}}G_{3})\bigsqcup(G_{1}\overleftarrow{\nabla}_{\widehat{\frac{n}{[d_1,d_2]}}}G_{4}),$$ where  $G_{1}=(\overleftrightarrow{K_{n}},\hat{n}),\ G_{2}=[\bigsqcup_{2}(\overleftrightarrow{K_{(d_{1}-1)n}},\widehat{\frac{n}{d_{1}}})],\ G_3=(\overleftrightarrow{K_{(d_{2}-1)n}},\widehat{\frac{n}{d_2}}),\ G_{4}=(\overleftrightarrow{K_{(d_{1}-1)(d_{2}-2)n}},\widehat{\frac{n}{[d_1,d_2]}})$, where $[d_1,d_2]$ is the least common multiple of $d_1$ and $d_2$.
	\end{theorem}
	\begin{proof}
		For a given pretzel link $P_1^3$ and a dihedral quandle $\mathbb{Z}_{n}$, by Theorem \ref{th3}, we have $|Hom(P_1^3,\mathbb{Z}_n)|=d_{1}d_2n$, where $d_{1}=\gcd(Q,n),d_2=\gcd(\frac{|P|}{Q},n),P=p_1p_2+p_1p_3+p_2p_3,Q=\gcd(p_{1},p_{2},p_3)$, which includes $n$ trivial colorings. It is easy to see that in this case, the quiver structure of $P_1^3$ contains a substructure $G_{1}=(\overleftrightarrow{K_{n}},\hat{n})$.\par 
		Next, we consider the case of the $(d_{1}d_2-1)n$ nontrivial colorings. Let $\phi _{a}\equiv \phi _{a_{1},a_{2},a_{3}},\phi _{b}\equiv \phi _{b_{1},b_{2},b_{3}}$ be any two quandle colorings in $Hom(P_1^3,\mathbb{Z}_n)$, i.e., $\phi _{a}(x_{i})\equiv a_{i},\phi _{b}(x_{i})\equiv b_{i},x_{i}\in \mathcal{Q}(P_1^3)$ where $a_{i},b_{i}\in \mathbb{Z}_{n},i=1,2,3$.\par 
		Without loss of generality, we assume that $\phi_{a}$ is a nontrivial quandle coloring. Then, from the proof of Theorem \ref{th3}, it follows that $d_{1}a_{1}\equiv d_{1}a_{2},d_2a_3\equiv d_{2}a_{2}$, and at least one of the inequalities $a_{1}\not \equiv a_{2},a_{1}\not \equiv a_{3},a_{2}\not \equiv a_{3}$ must hold. Otherwise, $\phi_{a}$ would be a trivial coloring.
		\begin{romanlist}[(i)]
			\item 	When $a_{2}\equiv a_{3}\not \equiv  a_{1},b_{2}\equiv b_{3}\not \equiv b_{1}$, from $d_{1}a_{1}\equiv d_{1}a_{2}$, we obtain $a_{2}\equiv a_{1}+k\frac{n}{d_{1}},k=1,2,...,d_{1}-1$. Similarly, it follows that $b_{2}\equiv b_{1}+\beta \frac{n}{d_{1}},\beta=1,2,...,d_{1}-1$.\par 
			Suppose there exists a homomorphism $f=f_{\alpha_{0}\alpha_{1}...\alpha_{n-1}}\in Hom(\mathbb{Z}_{n},\mathbb{Z}_{n})$ such that $f\circ \phi _{a}=\phi _{b}$, i.e., $f(a_{1})\equiv b_{1},f(a_{1}+k\frac{n}{d_{1}})\equiv b_{1}+\beta \frac{n}{d_{1}}$. Then, there exists $\tau \in \mathbb{Z}_{n}$ satisfying the following system of equations:
			\begin{center}
				\begin{tabular}{l}
					$f_{\alpha_{0}\alpha_{1}...b_{1}\tau...\alpha_{n-1}}(a_{1})\equiv b_{1},$\\
					$f_{\alpha_{0}\alpha_{1}...b_{1}\tau...\alpha_{n-1}}(a_{1}+1)\equiv \tau,$\\
					$f_{\alpha_{0}\alpha_{1}...b_{1}\tau...\alpha_{n-1}}(a_{1}+2)\equiv 2\tau -b_{1},$\\
					$f_{\alpha_{0}\alpha_{1}...b_{1}\tau...\alpha_{n-1}}(a_{1}+3)\equiv 3\tau -2b_{1},$\\
					\qquad\;\;\ ...\\
					$f_{\alpha_{0}\alpha_{1}...b_{1}\tau...\alpha_{n-1}}(a_{1}+k \frac{n}{d_{1}})\equiv k\frac{n}{d_{1}} \tau -(k\frac{n}{d_{1}}-1)b_{1},$\\
				\end{tabular}
			\end{center}
			Since $f(a_{1}+k\frac{n}{d_{1}})\equiv b_{1}+\beta \frac{n}{d_{1}}$, we have: $b_{1}+\beta \frac{n}{d_{1}}\equiv k\frac{n}{d_{1}} \tau -(k\frac{n}{d_{1}}-1)b_{1}$. Simplifying the equation, we obtain:
			\begin{equation}
				k\frac{n}{d_{1}}\tau \equiv k\frac{n}{d_{1}}b_{1}+\beta \frac{n}{d_{1}}.
				\label{d4.1}
			\end{equation}
			Let $A=k\frac{n}{d_{1}},B=k\frac{n}{d_{1}}b_{1}+\beta \frac{n}{d_{1}}$, the equation (\ref{d4.1}) has solutions if and only if $\gcd(A,n)|B$. If it is solvable, the number of distinct solutions modulo $n$ is exactly equal to $\gcd (A,n)$.
			Clearly, we have:
			$$\gcd(A, n) = \gcd\left(k\frac{n}{d_1}, n\right) = \frac{n}{d_1}\gcd(k, d_1)$$
			and
			\begin{align*}
				\gcd(A,n)|B=&\frac{n}{d_1}\gcd(k, d_1) \;\Big|\; \frac{n}{d_1}(kb_1 + \beta)\\
				\iff &\gcd(k, d_1) \mid (kb_1 + \beta)\\
				\iff&\gcd(k, d_1) \mid \beta
			\end{align*}
			Since $d_1$ is a prime number, and $k=1,2,...,d_1-1$, it follows that $gcd(k,d_1)=1$. Consequently, $\gcd(k, d_1) \mid \beta$ holds, and the congruence equation (\ref{d4.1}) has $\frac{n}{d_1}$ solutions. \par
			Since the homomorphism f is uniquely determined by $b_{1}$ and $\tau$, where $b_{1}$ is an arbitrary fixed value, $\tau \in\mathbb{Z}_{n}$, there exist exactly $\frac{n}{d_1}\gcd (k, d_1)$ distinct homomorphisms $f\in Hom(\mathbb{Z}_{n},\mathbb{Z}_{n})$ satisfying $f\circ \phi _{a}=\phi _{b}$. Meanwhile, since $a_{2}\in \mathbb{Z}_{n}$ there are $(d_{1}-1)n$ distinct possible values for $a_{1}$ or $a_{3} $, i.e., in this case, there are $(d_{1}-1)n$ quandle coloring vertices. Therefore, $\mathcal{Q}_{\mathbb{Z}_{n}}(P_1^3)$ contains a subgraph $(\overleftrightarrow{K_{(d_{1}-1)n}},\widehat{\frac{n}{d_1}})$.
			\item When $a_{1}\equiv a_{2}\not \equiv a_{3},b_{1}\equiv b_{2}\not \equiv b_{3}$, we can obtain $a_{3}\equiv a_{2}+l\frac{n}{d_{2}},l=1,2,...,d_{2}-1$ and similarly, $b_{3}\equiv b_{2}+\beta^{\prime}\frac{n}{d_{2}},\beta^{\prime}=1,2,...,d_{2}-1$. Following analogous reasoning to the previous case, we can also contains a subgraph $G_3=(\overleftrightarrow{K_{(d_{2}-1)n}},\widehat{\frac{n}{d_2}})$ of $\mathcal{Q}_{\mathbb{Z}_{n}}(P_1^3)$.
			\item When $a_{1}\equiv a_{3}\not \equiv a_{2},b_{1}\equiv b_{3}\not \equiv b_{2}$, similarly to the above case, we can deduce that there are also $(d_{1}-1)n$ non-trivial quandle colorings at the vertices in this scenario. And we have
			\begin{equation}
				\begin{cases}
					k\frac{n}{d_{1}}\equiv l\frac{n}{d_{2}},\qquad \ \,(1a)\\
					\beta \frac{n}{d_{1}}\equiv \beta^{\prime}\frac{n}{d_{2}}.\qquad(1b)
				\end{cases}
				\label{4.2}
			\end{equation}
			\begin{equation}
				\begin{cases}
					k\frac{n}{d_{1}}\tau \equiv k\frac{n}{d_{1}}b_{2}+\beta \frac{n}{d_{1}},\qquad k,\beta =1,2,...,d_{1}-1,\qquad(2a)\\
					l\frac{n}{d_{1}}\tau \equiv l\frac{n}{d_{1}}b_{2}+\beta^{\prime} \frac{n}{d_{1}}. \qquad l,\beta^{\prime}=1,2,...,d_{2}-1.\qquad(2b)
				\end{cases}
				\label{4.3}
			\end{equation}
			First, we examine the system of equations (\ref{4.2}):
			\begin{align*}
				(1a)
				&\iff there\ exists\ an\ integer\ m\in \mathbb{Z}\ such \ that\ k\frac{n}{d_{1}}-l\frac{n}{d_{1}}=mn\\
				&\iff(\frac{k}{d_{1}}-\frac{l}{d_{2}})n=mn
			\end{align*} \\
			Since $k=1,2,...,d_{1}-1,l=1,2,...,d_{2}-1$ the only possibility is $\frac{k}{d_{1}}-\frac{l}{d_{2}}=0$,i.e.,
			\begin{equation}
				\frac{k}{d_{1}}=\frac{l}{d_{2}}.
				\label{4.4}
			\end{equation}
			Similarly, from (1b) we obtain: 
			\begin{equation}
				\frac{\beta }{d_{1}}=\frac{\beta^{\prime}}{d_{2}}.
				\label{4.5}
			\end{equation}
			Thus, the system of equations (\ref{4.3}) is equivalent to either (2a) or (2b). This conclusion appears strange,since (2a) and (2b) evidently yield distinct quiver structures. Consequently, the system (\ref{4.2}) warrants further investigation to resolve this apparent discrepancy.\par 
			Without loss of generality, we assume that $d_{1}<d_{2}$, if $d_{1}\nmid d_{2}$ then there exist integers $r<d_{1}$ and $q$ such that $d_{2}=qd_{1}+r$, substituting this into equation (\ref{4.4}) yields:
			\begin{align*}
				\frac{k}{d_{1}}=\frac{l}{qd_{1}+r}
				&\iff kqd_{1}+kr=d_{1}l\\
				&\iff (l-kq)d_{1}=kr
			\end{align*}
			Since $\gcd(k,d_{1})=1$, it follows that $r=0$ and $d_1|d_2$, i.e., $d_{1}=1$. Then from (\ref{4.4}) we obtain $qk=l$, and consequently, the system (\ref{4.3}) is equivalent to (2a). It then follows from (i) that $\mathcal{Q}_{\mathbb{Z}_{n}}(P_1^3)$ contains a subgraph $(\overleftrightarrow{K_{(d_{1}-1)n}},\widehat{\frac{n}{d_1}})$.
			\item When $a_{1}\not \equiv a_{2}\not \equiv a_{3},b_{1}\not \equiv b_{2}\not \equiv b_{3}$, combining the above discussion, we obtain $a_{1}=a_{2}+k\frac{n}{d_{1}},a_{3}= a_{2}+l\frac{n}{d_{1}}$ and $b_{1}=b_{2}+\beta \frac{n}{d_{1}},b_{3}=b_{2}+\beta ^{\prime}\frac{n}{d_{2}}$, where $k,l,\beta,\beta ^{\prime}=1,2,...,d_{1}-1,k\neq l,\beta \neq \beta^{\prime}$. At this stage, only $(d_{1}-1)(d_{2}-2)n$ non-trivial quandle coloring vertices remain to be analyzed. By employing the methods from (i), we obtain the following two systems of equations:
			\begin{equation}
				\begin{cases}
					k\frac{n}{d_{1}}\not \equiv l\frac{n}{d_{2}},\qquad(3a)\\
					\beta \frac{n}{d_{1}}\not \equiv \beta^{\prime}\frac{n}{d_{2}}.\qquad(3b)
				\end{cases}
				\label{4.6}
			\end{equation}
			\begin{equation}
				\begin{cases}
					k\frac{n}{d_{1}}\tau \equiv k\frac{n}{d_{1}}b_{2}+\beta \frac{n}{d_{1}}\qquad k,\beta =1,2,...,d_{1}-1,\\
					l\frac{n}{d_{1}}\tau \equiv l\frac{n}{d_{1}}b_{2}+\beta^{\prime} \frac{n}{d_{1}} \qquad l,\beta^{\prime}=1,2,...,d_{2}-1.
				\end{cases}
				\label{4.7}
			\end{equation}
			Since $\gcd(\frac{n}{d_{1}},n)=\frac{n}{d_{1}},\gcd(\frac{n}{d_{2}},n)=\frac{n}{d_{2}}$, we obtain:
			\begin{equation}
				\begin{cases}
					k\tau \equiv kb_{2}+\beta \quad (mod\ d_{1}),\\
					l\tau \equiv lb_{2}+\beta^{\prime}\quad(mod\ d_{2}).
				\end{cases}
				\label{4.8}
			\end{equation}
			Let $[d_{1},d_{2}]$ denote the least common multiple of $d_{1}$ and $d_{2}$, then there exist $m_{1},m_{2}\in \mathbb{Z}^{+}$ such that $[d_{1},d_{2}]=m_{1}d_{1}=m_{2}d_{2}$ where $\gcd(m_{1},m_{2})=1$. According to \cite{Stein2017}, the system of equations (\ref{4.8}) is equivalent to the system (\ref{4.9})  below.
			\begin{equation}
				\begin{cases}
					km_{1}\tau \equiv km_{1}b_{2}+m_{1}\beta \\
					lm_{2}\tau \equiv lm_{2}b_{2}+m_{2}\beta^{\prime} 
				\end{cases}\quad (mod\ [d_{1},d_{2}]).
				\label{4.9}
			\end{equation}
			Since $\gcd(k,d_{1})=\gcd(k,d_{2})=1$, the systems of equations (\ref{4.8}) and (\ref{4.9}) have a unique solution if and only if the following condition (\ref{4.10}) holds. Moreover, in this case, the system of equations (\ref{4.7}) has $\frac{n}{[d_{1},d_{2}]}$ solutions.
			\begin{equation}
				k^{\phi (d_{1})-1}\beta \equiv l^{\phi (d_{2})-1}\beta^{\prime}\pmod{\gcd\left(\frac{d_1}{\gcd(k, d_1)}, \frac{d_2}{\gcd(l, d_2)}\right)}
				\label{4.10}
			\end{equation}\par
			Here, $\phi(d_{i}),i=1,2$ denotes the Euler's totient function. Clearly, when $k=\beta ,l=\beta ^{\prime}$, the equality (\ref{4.10}) holds. Combining this with system (\ref{4.6}) and (iii), there are $(d_{1}-1)(d_{2}-2)$ possible ways for the equality to be satisfied. \par
			It is straightforward to verify that neither $k=\beta ,l\neq \beta ^{\prime}$ nor $k\neq \beta ,l=\beta^{\prime}$ can exist. When $k\neq \beta ,l\neq \beta ^{\prime}$, for fixed $k$ and $l$ there are $d_{1}-1$ possible choices for $\beta$ and $d_{2}-2$ choices for $\beta ^{\prime}$. When $d_{1}\neq d_{2}$, this implies that the $(d_{1}-1)(d_{2}-2)$ quandle-colored vertices cannot be further partitioned. Consequently, in this case, $\mathcal{Q}_{\mathbb{Z}_{n}}(P_1^3)$ contains a subgraph $G_{4}=(\overleftrightarrow{K_{(d_{1}-1)(d_{2}-2)n}},\widehat{\frac{n}{[d_{1},d_{2}]}})$.
			
		\end{romanlist}
		\par 
		We now examine the connection relationships between the aforementioned subgraphs: Since no mapping in $Hom(\mathbb{Z}_{n},\mathbb{Z}_{n})$ can send a single value in $\mathbb{Z}_{n}$ to two distinct values in $\mathbb{Z}_{n}$, there exist no directed edges from the subgraph $G_{1}$ (containing only trivially colored vertices) pointing to other subgraphs containing non-trivially colored vertices. Moreover, it is straightforward to verify that no directed edges connect the subgraphs generated in cases (i), and (iii). Combining these subgraphs yields $G_{2}=[\bigsqcup_{2}(\overleftrightarrow{K_{(d_{1}-1)n}},\widehat{\frac{n}{d_{1}}})]$. By the same reasoning, we conclude that there are no directed edges from $G_{2}$ or $G_{3}$ to $G_{4}$ and $G_{2}$ to $G_{3}$. We now examine the directed edges from $G_{2}$, $G_{3}$ and $G_{4}$ to $G_{1}$, as well as those from $G_{4}$ to $G_{2}$ and $G_{3}$.
		\begin{enumerate}[label=(\arabic*)]
			\item First, we consider the directed edges from $G_{2}$, $G_{3}$ and $G_{4}$ to $G_{1}$. Let $\phi_{a}=\phi _{a_{1}a_{2}a_{3}},\phi_{b}=\phi_{b_{1}b_{2}b_{3}}$, where $v_{\phi_{a}}\in G_{2},v_{\phi_{b}}\in G_{1}$. Using the method of (i), we can obtain:
			\begin{equation}
				k\frac{n}{d_{1}}\tau \equiv k\frac{n}{d_{1}}b_{2}.
				\label{4.11}
			\end{equation}
			Since $\gcd(k,d_1)=1$, it is easy to know the congruence equation (\ref{4.11}) have $\frac{n}{d_{1}}$ solutions. The relationship between $G_{2}$ and $G_{1}$ is $G_{1}\overleftarrow{\nabla}_{\hat{\frac{n}{d_{1}}}}G_{2}$.\\
			\item Similarly, the relationship between $G_{3}$ and $G_{1}$ is $G_{1}\overleftarrow{\nabla}_{\hat{\frac{n}{d_{2}}}}G_{3}$.
			\item Let $v_{\phi_a}\in G_{4},v_{\phi _{b}}\in G_{1}$. We have:
			\begin{equation}
				\begin{cases}
					k\frac{n}{d_{1}}\tau \equiv k\frac{n}{d_{1}}b_{2},\\
					l\frac{n}{d_{2}}\tau \equiv l\frac{n}{d_{2}}b_{2}.
				\end{cases}
				\label{4.12}
			\end{equation}
			Consequently, we have:
			\begin{equation}
				\begin{cases}
					\tau \equiv b_{2}\qquad (mod\ d_{1}),\\
					\tau \equiv b_{2}\qquad (mod\ d_{2}).
				\end{cases}
				\label{4.13}
			\end{equation}
			Then we will get
			\begin{equation}
				\tau\equiv b_2\quad (mod\ [d_1,d_2])
			\end{equation}
			Therefore, the system of equations (\ref{4.12}) has $\frac{n}{[d_1,d_2]}$ solutions. Consequently, the relationship between $G_{4}$ and $G_{1}$ is given by: $G_{1}\overleftarrow{\nabla}_{\hat{\frac{n}{[d_1,d_2]}}}G_{4}$.
			\item Finally, examining the directed edge from $G_{3}$ to $G_{2}$, we let $\phi_{a}=\phi _{a_{1}a_{2}a_{3}},\phi_{b}=\phi_{b_{1}b_{2}b_{3}}$ where $v_{\phi _{a}}\in G_{3},v_{\phi_b}\in G_{2}$. Using the same method as in (i), we obtain the following three cases of system of congruences:
			\begin{equation}
				\begin{cases}
					k\frac{n}{d_{1}}\tau \equiv k\frac{n}{d_{1}}b_{1}+\beta \frac{n}{d_{1}},\\
					l\frac{n}{d_{1}}\tau \equiv l\frac{n}{d_{1}}b_{1}.
				\end{cases}
				\label{d4.11}
			\end{equation}
			\begin{equation}
				\begin{cases}
					k\frac{n}{d_{1}}\tau \equiv k\frac{n}{d_{1}}b_{1},\\
					l\frac{n}{d_{1}}\tau \equiv l\frac{n}{d_{1}}b_{1}+\beta ^{\prime}\frac{n}{d_{1}}.
				\end{cases}
				\label{d4.12}
			\end{equation}
			\begin{equation}
				\begin{cases}
					k\frac{n}{d_{1}}\tau \equiv k\frac{n}{d_{1}}b_{1}+\beta \frac{n}{d_{1}},\\
					l\frac{n}{d_{1}}\tau \equiv l\frac{n}{d_{1}}b_{1}+\beta \frac{n}{d_{1}}.
					\label{d4.13}
				\end{cases}
			\end{equation}
			Since $\gcd(k,d_{1})=\gcd(\beta,d_{1})=\gcd(l,d_{1})=1$, it is clear that system of congruences (\ref{d4.11}) and (\ref{d4.12}) have no solution. If system of congruences (\ref{d4.13}) has a solution, it must satisfy $k^{\phi (n)-1}\beta =l^{\phi (n)-1}\beta $, which implies $k=l$, but, this contradicts the condition $v_{\phi _{a}}\in G_{3}$. Therefore, there are no directed edges from $G_{3}$ to $G_{2}$ which indicates that there are no directed edges from $G_{4}$ to $G_{2}$. Similarly, there are no directed edges from $G_{4}$ to $G_{3}$.\par 
		\end{enumerate}
		In summary, we can conclude that:\\ $\mathcal{Q}_{\mathbb{Z}_{n}}(P_1^3)=(G_{1}\overleftarrow{\nabla}_{\widehat{\frac{n}{d_{1}}}}G_{2})\bigsqcup (G_{1}\overleftarrow{\nabla}_{\widehat{\frac{n}{d_{2}}}}G_{3})\bigsqcup(G_{1}\overleftarrow{\nabla}_{\widehat{\frac{n}{[d_1,d_2]}}}G_{4})$.
	\end{proof}
	
	According to Theorem \ref{thmp}, for a given dihedral quandle $\mathbb{Z}_{n}$, when $n$ is a prime number, the coloring number of n-pretzel links is $n^q,q=1,2,...,m$, from \cite{Zhou2023,Zhou2024,Elhamdadi2025}, this indicates that the solution to the coloring equations hasindicates that the solution to the coloring equations has $q$ free variables. Therefore, we can derive the quiver structure for general pretzel links.
	\begin{theorem}
		Given a pretzel link $P^m=(p_{1},p_{2},...,p_{m})$, and a dihedral quandle $\mathbb{Z}_{n}$ where $n$ is a prime number. Let $h=\sum p_{i_{1}}p_{i_{2}}...p_{i_{m-1}},1\leq i_{1} \leq i_{1}\leq ...\leq i_{m-1}$. The quandle coloring quivers if $P^m$ is:
		\begin{equation}
			\mathcal{Q}_{\mathbb{Z}_{n}}(P^m)=
			\begin{cases}
				(\overleftrightarrow{K_{n}},\hat{n})\qquad h\not \equiv 0,\\
				(\overleftrightarrow{K_{n}},\hat{n})\overleftarrow{\nabla}_{\hat{1}}(\overleftrightarrow{K_{(n-1)n}},\hat{1}) \quad h\equiv 0\ and\ p_{1}...p_{m-2}\equiv ...\equiv p_{3}...p_{m}\equiv 0\\ \qquad \qquad \qquad \qquad \qquad \quad do\ not\ all\ hold,\\
				\qquad \quad\;\ \, \vdots\\   
				(\overleftrightarrow{K_{n}},\hat{n})\overleftarrow{\nabla}_{\hat{1}}[\bigsqcup_{\frac{n^{m-1}-n}{n(n-1)}}(\overleftrightarrow{K_{n(n-1)}},\hat{1})] \quad among\ p_{1},p_{2},...,p_{m}\ only\ one\\\qquad\qquad\qquad\qquad\qquad\qquad\qquad\quad\;\  is\ coprime\ with\ n,\\
				(\overleftrightarrow{K_{n}},\hat{n})\overleftarrow{\nabla}_{\hat{1}}[\bigsqcup_{\frac{n^{m}-n}{n(n-1)}}(\overleftrightarrow{K_{n(n-1)}},\hat{1})] \qquad p_{1}\equiv p_{2}\equiv ...\equiv p_{m}\equiv 0.
			\end{cases}
			\nonumber
		\end{equation}
	\end{theorem}
	\newpage
	\begin{example}
		Given a pretzel link $P_{366}=(3,3,6)$, and a dihedral quandle $\mathbb{Z}_{n}$ with $n=6$, it follows from Theorem \ref{th3} that the number of colorings of the oriented link $P_{336}$ is $Hom(P_{336},\mathbb{Z}_{n})=3\times3\times n=54$. Consequently, by Theorem \ref{th4.1}, the quiver structure of $P_{336}$ is
		\begin{equation}
			\mathcal{Q}_{\mathbb{Z}_{6}}(P_{336})=G_{1}\overleftarrow{\nabla}_{\hat{2}}G_{2}
			\nonumber
		\end{equation}
		where $G_{1}=(\overleftrightarrow{K_{6}},\hat{6}),G_{2}=[\bigsqcup_{4}(\overleftrightarrow{K_{12}},\hat{2})].$\par 
		To visualize the obtained quiver structure, we first construct the following oriented pseudographs: $(\overleftrightarrow{K_{6}},\hat{6}),(\overleftrightarrow{K_{12}},\hat{2})$.\par 
		\begin{figure}[H]
			\centering
			\includegraphics[width=0.58\textwidth,height=0.5\textwidth]{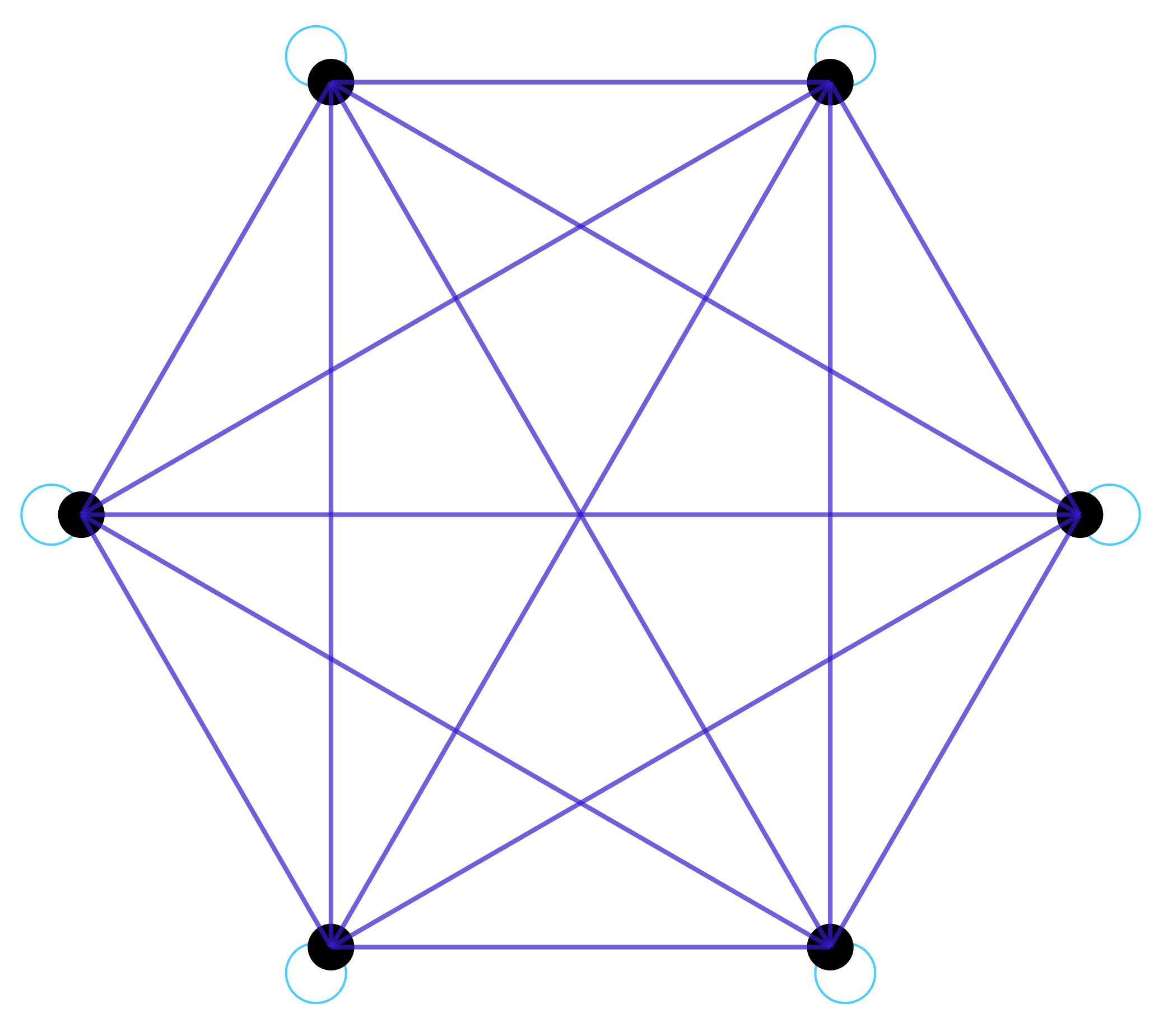}
			\caption{$(\overleftrightarrow{K_{6}},\hat{6})$}
		\end{figure}
		For clarity of presentation, in the figure above we have omitted the repeated edges and the orientations the edges in $(\overleftrightarrow{K_{6}},\hat{6})$. In the diagram of $(\overleftrightarrow{K_{6}},\hat{6})$, each edge has a weight of 6, indicating that there are 6 edges between every pair of vertices, each with an arrow denoting its direction. Similarly, for visual simplicity, the orientations of the edges are omitted in the subsequent figure $(\overleftrightarrow{K_{12}},\hat{2})$.
		\begin{figure}[H]
			\centering
			\includegraphics[width=0.64\textwidth,height=0.61\textwidth]{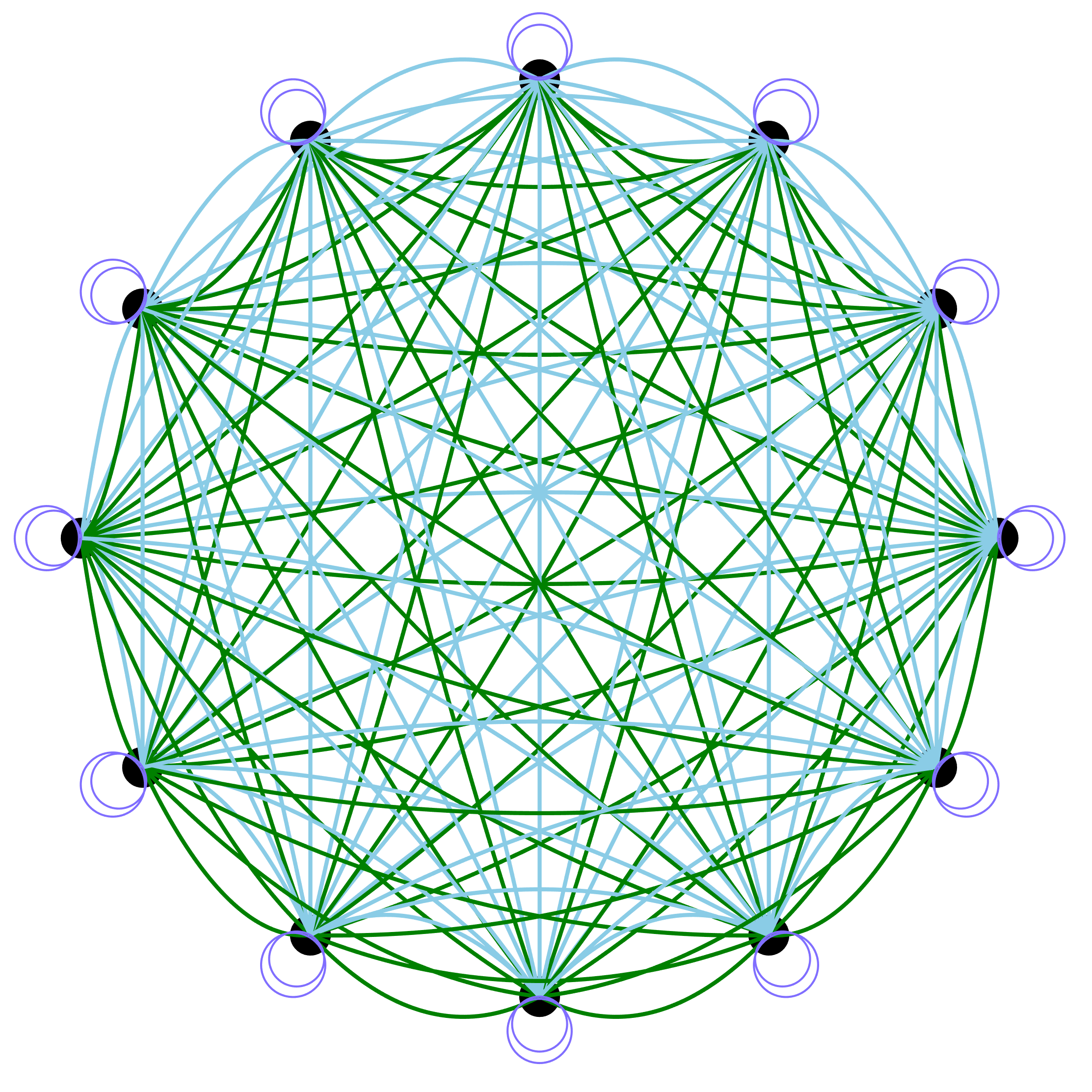}
			\caption{$(\overleftrightarrow{K_{12}},\hat{2})$}
		\end{figure}
		If we represent all vertices in $(\overleftrightarrow{K_{6}},\hat{6})$ with black dots and all vertices in $(\overleftrightarrow{K_{12}},\hat{2})$ with green dots, then the relationship between $G_{1}$ and $G_{2}$ can be expressed as:
		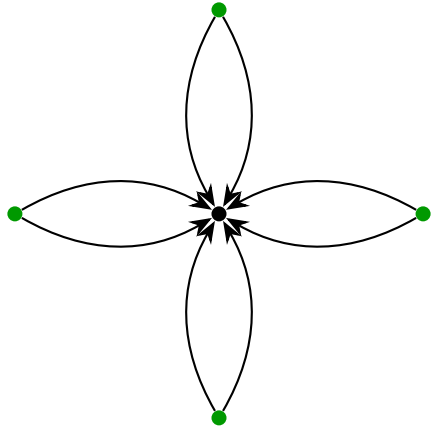
\begin{figure}[H]
			\centering
			\begin{tikzpicture}[>={Stealth[length=3mm]}]
				\node[circle, fill=green!60!black, inner sep=2pt] (left) at (-2.7,0) {};
				
				\node[circle, fill=green!60!black, inner sep=2pt] (below) at (0,-2.7) {};
				\node[circle, fill=green!60!black, inner sep=2pt] (right) at (2.7,0) {};
				\node[circle, fill=green!60!black, inner sep=2pt] (above) at (0,2.7) {};
				\node[circle, fill=black, inner sep=2pt] (center) at (0,0) {};
				
				\draw[->, thick] (left) to [out=30, in=150] (center);  
				\draw[->, thick] (left) to [out=-30, in=-150] (center); 
				
				\draw[->, thick] (above) to[bend left=30] (center);
				\draw[->, thick] (above) to[bend right=30] (center);
				
				\draw[->, thick] (right) to [out=150, in=30] (center);   
				\draw[->, thick] (right) to [out=-150, in=-30] (center); 
				\draw[->, thick] (below) to[bend left=30] (center);
				\draw[->, thick] (below) to[bend right=30] (center);
			\end{tikzpicture}
			\caption{$(\overleftrightarrow{K_{6}},\hat{6})\overleftarrow{\nabla}_{\hat{2}}[\bigsqcup_{4}(\overleftrightarrow{K_{12}},\hat{2})$}
			\label{t45}
		\end{figure}\par 
		
	\end{example}

	\bibliographystyle{ws-jktr}
	\bibliography{sample}

@article{Gabriel1972,
  title={Unzerlegbare darstellungen I},
  author={Gabriel, Peter},
  journal={Manuscripta mathematica},
  volume={6},
  pages={71--103},
  year={1972},
  publisher={Springer}
}

@article{Kathryn,
   title={Fundamentally Different \textit{m}-Colorings of Pretzel Knots},
  author={Brownell, Kathryn and O’NEIL, KAITLYN and Taalman, Laura},
  year={2004},
  publisher={Citeseer}
}

@book{Florentin,
   title={Algorithms for solving linear congruences and systems of linear congruences},
  author={Smarandache, Florentin},
  year={1987},
  publisher={Infinite Study}
}

@book{Stein2017,
   title = {Elementary Number Theory: Primes, Congruences, and Secrets},
   author={William Stein},
   year = {2009},
   publisher={Springer,New York}
}

@techReport{Wang,
   title={Researches on the Elementary Modular Matrix Transformations and the System of Linear Congruence Equations},
  author={Daiwei, Wang},
  url={https://docslib.org/doc/5171427/researches-on-the-elementary-modular-matrix-transformations-and-the}
}

@article{Elhamdadi2025,
   author = {Mohamed Elhamdadi and Brooke Jones and Minghui Liu},
   doi = {10.1007/s10801-024-01373-4},
   issn = {0925-9899},
   issue = {1},
   journal = {Journal of Algebraic Combinatorics},
   month = {2},
   pages = {4},
   title = {Quandle coloring quivers of general torus links by dihedral quandles},
   volume = {61},
   url = {https://link.springer.com/10.1007/s10801-024-01373-4},
   year = {2025}
}

@book{mathews1892,
  title={Theory of Numbers: Part I.},
  author={Mathews, George Ballard},
  year={1892},
  publisher={Deighton, Bell and Company}
}

@article{Zhou2024,
   author = {Boxin Zhou and Ximin Liu},
   doi = {10.1142/S0218216524500147},
   issn = {17936527},
   issue = {5},
   journal = {Journal of Knot Theory and its Ramifications},
   month = {4},
   publisher = {World Scientific},
   title = {Quandle coloring quivers of \textit{(p, q)}-torus links},
   volume = {33},
number={05},
year = {2024},
pages={2450014}  
}

@article{Basi2021p2,
   author = {Jagdeep Basi and Carmen Caprau},
   journal={Journal of Knot Theory and Its Ramifications},
   doi = {10.1142/S0218216522500572},
   month = {12},
   title = {Quandle Coloring Quivers of \textit{(p, 2)}-Torus Links},
   url = {http://arxiv.org/abs/2112.05297 http://dx.doi.org/10.1142/S0218216522500572},
volume={31},
  number={09},
  pages={2250057},
   year = {2021},
   publisher={World Scientific}
}

@article{Cho2018,
  title={Quandle coloring quivers},
  author={Cho, Karina and Nelson, Sam},
  journal={Journal of Knot Theory and Its Ramifications},
  volume={28},
  number={01},
  pages={1950001},
  year={2019},
  publisher={World Scientific}
}

@book{Basi2021,
   title={Quandle coloring quivers of torus knots},
  author={Basi, Jagdeep},
  year={2021},
  month={5},
  school={California State University},
  publisher={California State University, Fresno},
  pages={1--73},
url={http://hdl.handle.net/20.500.12680/th83m459q}
}

@article{Zhou2023,
   author = {Boxin Zhou and Ximin Liu},
   doi = {10.1142/S0218216523500165},
   issn = {02182165},
   issue = {3},
   journal = {Journal of Knot Theory and its Ramifications},
   month = {3},
   publisher = {World Scientific},
   title = {Quandle coloring quivers of \textit{(p,3)}-Torus links},
   volume = {32},
   year = {2023}
}

@article{Joyce1982,
  title={A classifying invariant of knots, the knot quandle},
  author={Joyce, David},
  journal={Journal of Pure and Applied Algebra},
  volume={23},
  number={1},
  pages={37--65},
  year={1982},
  publisher={Elsevier}
}

@article{Kim2007,
  title={Some invariants of pretzel links},
  author={Kim, Dongseok and Lee, Jaeun},
  journal={Bulletin of the Australian Mathematical Society},
  volume={75},
  number={2},
  pages={253--271},
  year={2007},
  publisher={Cambridge University Press}
}

@article{Mohamed2015,
   title={Quandles: An Introduction to the algebra of Knots, the Student Mathematical Library, vol. 74},
  author={Nelson, Sam and Elhamdadi, Mohammed},
  journal={American Mathematical Society},
volume={74},
isbn={978-1-4704-2213-4},
  year={2015}
}

@inproceedings{Brownell2003,
  title={Fundamentally Different \textit{m}-Colorings of Pretzel Knots},
  author={Kathryn Cramer Brownell},
  year={2003},
  publisher={Citeseer},
url={https://api.semanticscholar.org/CorpusID:38744286}
}

@mastersthesis{Ostrander,
   title={\textit{P}-Coloring of pretzel knots},
  author={Ostrander, Robert},
  year={2013},
  type={Master of Science},
  school={California State University, Northridge}
}
\end{document}